\documentclass[a4paper]{amsart}

\pdfoutput=1

\usepackage{amsmath,amsthm,amssymb}
\usepackage{mathtools}
\usepackage{mathrsfs}

\usepackage{microtype}
\usepackage{lmodern}
\usepackage[T1]{fontenc}
\usepackage{dsfont}
\usepackage[english]{babel}
\usepackage{hyperref}

\usepackage{pbox}

\usepackage{csquotes}

\usepackage{multirow}
\usepackage{enumerate}

\usepackage{tikz,pgffor}

\allowdisplaybreaks

\newtheorem{theorem}{Theorem}[section]
\newtheorem{lemma}[theorem]{Lemma}
\newtheorem{prop}[theorem]{Proposition}
\newtheorem{cor}[theorem]{Corollary}
\newtheorem{conjecture}[theorem]{Conjecture}
\theoremstyle{definition}
\newtheorem{remark}[theorem]{Remark}

\numberwithin{equation}{section}
\DeclareMathOperator{\cone}{cone}

\DeclareMathOperator{\Gal}{Gal}
\DeclareMathOperator{\Pic}{Pic}

\DeclareMathOperator{\Eff}{Eff}

\DeclareMathOperator{\vol}{vol}

\DeclareMathOperator*{\Hom}{Hom}

\DeclareMathOperator*{\reg}{reg}

\DeclareMathOperator*{\sing}{sing}

\DeclareMathOperator*{\Spec}{Spec}
\DeclareMathOperator*{\Proj}{Proj}

\DeclareMathOperator*{\Cl}{Cl}
\newcommand{\err}{R}

\DeclareMathOperator{\rk}{rk}

\newcommand{\rleft}{\mathopen{}\mathclose\bgroup\left}
\newcommand{\rright}{\aftergroup\egroup\right}

\newcommand{\n}{n}

\newcommand{\Qd}{\mathds{Q}}
\newcommand{\Qbar}{\overline{\Qd}}
\newcommand{\Kc}{\Qbar}

\newcommand{\Vd}{\mathds{V}}
\newcommand{\Zd}{\mathds{Z}}
\newcommand{\Fd}{\mathds{F}}
\newcommand{\Rd}{\mathds{R}}

\newcommand{\Gd}{\mathds{G}}
\newcommand{\Pd}{\mathds{P}}

\newcommand{\Ad}{\mathds{A}}

\newcommand{\SL}{\mathrm{SL}}

\newcommand{\Z}{\Zd}

\newcommand{\Vm}{\mathcal{V}}
\newcommand{\Cm}{\mathcal{C}}
\newcommand{\Fm}{\mathcal{F}}

\newcommand{\Dm}{\mathcal{D}}
\newcommand{\Rm}{\mathcal{R}}

\newcommand{\Nm}{\mathcal{N}}
\newcommand{\Mm}{\mathcal{M}}
\newcommand{\Lm}{\mathcal{L}}
\newcommand{\Um}{\mathcal{U}}
\newcommand{\Om}{\mathcal{O}}

\newcommand{\Tm}{\mathcal{T}}

\newcommand{\Xf}{\mathfrak{X}}
\newcommand{\Yf}{\mathfrak{Y}}
\newcommand{\afr}{\mathfrak{a}}
\newcommand{\bfr}{\mathfrak{b}}
\newcommand{\cfr}{\mathfrak{c}}

\newcommand{\Ms}{\mathscr{M}}

\newcommand{\ie}{i.\,e.,~}
\newcommand{\eg}{e.\,g.,~}

\newcommand{\df}{\mathrm{d}}
\newcommand{\fibphi}{\phi^{-1}}
\newcommand{\fibphix}{\fibphi(x)}
\newcommand{\fibphib}{\phi'^{-1}}
\newcommand{\fibphibx}{\fibphib(x)}

\newcommand{\tX}{\widetilde{X}}
\newcommand{\hX}{\widehat{X}}
\newcommand{\tXf}{\widetilde{\Xf}}
\newcommand{\stX}{\smash{\tX}}
\newcommand{\shX}{\smash{\hX}}
\newcommand{\stXf}{\smash{\tXf}}
\newcommand{\tY}{\widetilde{Y}}
\newcommand{\tYf}{\widetilde{\Yf}}
\newcommand{\stY}{\smash{\tY}}
\newcommand{\stYf}{\smash{\tYf}}

\newcommand{\efr}{\mathfrak{e}}
\newcommand{\ma}{\alpha}
\newcommand{\mb}{\beta}
\newcommand{\mc}{\gamma}
\newcommand{\cond}{\mathscr{H}}
\newcommand{\ca}{\hat{a}}
\newcommand{\cc}{\hat{c}}
\newcommand{\cy}{\hat{y}}

\makeatletter
\newcommand{\subalign}[1]{%
  \vcenter{%
    \Let@ \restore@math@cr \default@tag
    \baselineskip\fontdimen10 \scriptfont\tw@
    \advance\baselineskip\fontdimen12 \scriptfont\tw@
    \lineskip\thr@@\fontdimen8 \scriptfont\thr@@
    \lineskiplimit\lineskip
    \ialign{\hfil$\m@th\scriptstyle##$&$\m@th\scriptstyle{}##$\crcr
      #1\crcr
    }%
  }
}
\makeatother
\makeatother

\AtBeginDocument{%
   \def\MR#1{}
}

\begin{document}

\title{Manin's conjecture for certain spherical threefolds}

\author{Ulrich Derenthal}

\address
{Institut f\"ur Algebra, Zahlentheorie und Diskrete Mathematik, Leibniz
  Universit\"at Hannover, Welfengarten 1, 30167 Hannover, Germany}
\email{derenthal@math.uni-hannover.de}

\author{Giuliano Gagliardi}

\address
{Raymond and Beverly Sackler School of Mathematical Sciences,
  Tel Aviv University, 6997801 Tel Aviv, Israel}
\email{giulianog@mail.tau.ac.il}

\date{August 10, 2018}


\begin{abstract}
  We prove Manin's conjecture on the asymptotic behavior of the number of
  rational points of bounded anticanonical height for a spherical threefold with
  canonical singularities and two infinite families of spherical threefolds with
  log terminal singularities. Moreover, we show that one of these families does
  not satisfy a conjecture of Batyrev and Tschinkel on the leading constant in
  the asymptotic formula. Our proofs are based on the universal torsor method,
  using Brion's description of Cox rings of spherical varieties.
\end{abstract}

\maketitle

\microtypesetup{protrusion=false}
\tableofcontents
\microtypesetup{protrusion=true}

\section{Introduction}

\subsection{Spherical varieties and Manin's conjecture}

Manin's conjecture \cite{MR89m:11060, MR1032922, MR1340296, MR1679843,
  MR2019019, pla} makes a precise prediction for the asymptotic behavior of the
number of rational points of bounded anticanonical height on (almost) Fano
varieties over number fields whose set of rational points is Zariski dense.

For a smooth Fano variety over $\Qd$ with a Zariski dense set of rational
points, one may introduce an anticanonical height function
$H\colon X(\Qd) \to \Rd_{>0}$ and ask for the asymptotic behavior of the number
of rational points of bounded height, as the height bound tends to infinity. The
total number might be dominated by points on \emph{accumulating} subvarieties
(or, more generally, accumulating \emph{thin subsets}, see \cite[\S
8]{MR2019019}), and hence it is more interesting to restrict to their complement
$U$. By \cite[Conjecture~B']{MR1032922}, we are lead to the expectation that
\begin{equation*}
  N_{U,H}(B) \coloneqq \#\{x \in U(\Qd) \colon H(x) \le B\} \sim \cfr B (\log B)^{\rho-1}
\end{equation*}
as $B \to \infty$, where $\rho$ is the Picard number of $X$. A conjecture for
the leading constant $\cfr$ is given by Peyre in \cite{MR1340296}. If $X$ is a
singular Fano variety with a \emph{crepant resolution} $\pi \colon \tX \to X$
(i.e., a desingularization with $\pi^*(-K_X) = -K_{\tX}$), then
\cite[Conjecture~C']{MR1032922} and \cite[5.1]{MR2019019} tell us that such an
asymptotic formula should hold with $\rho$ and $\cfr$ computed on $\tX$. If $X$
has worse singularities, \cite[Conjecture~C']{MR1032922} and
\cite[3.6]{MR2019019} predict
\begin{equation*}
  N_{U,H}(B) \sim \cfr B^\afr(\log B)^{\bfr-1},
\end{equation*}
where we may have $\afr > 1$; Batyrev and Tschinkel \cite{MR1679843} give a
prediction for $\cfr$.

Manin's conjecture has been proved for some classes of varieties and several
individual examples. Most of the known cases are proved using either harmonic
analysis on adelic points or the universal torsor method combined with various
analytic techniques.

Many of them are \emph{spherical varieties}, i.e., normal $G$-varieties
containing a dense $B$-orbit, where $G$ is a connected reductive group and
$B \subseteq G$ is a Borel subgroup. Spherical varieties are a huge class of
varieties that admit a combinatorial description by spherical systems (Luna's
program \cite{lun01}) and colored fans (Luna--Vust theory \cite{lv83})
generalizing the combinatorial description of toric varieties.

In particular, harmonic analysis has been used to prove Manin's conjecture for
some classes of equivariant compactifications of algebraic groups, for example
flag varieties \cite{MR89m:11060}, toric varieties \cite{MR1620682},
horospherical varieties \cite{MR1723811}, and wonderful compactifications of
semi-simple groups \cite{MR2482443, MR2328719}. All these varieties are
spherical varieties; more precisely, flag varieties and toric varieties are
special cases of horospherical varieties (which are toric bundles over flag
varieties, at least after blow-ups); wonderful compactifications of semi-simple
groups are special cases of wonderful varieties. This approach has also been
applied to some non-spherical varieties, namely equivariant compactifications of
vector groups \cite{MR1906155} and Cayley's singular ruled cubic surface
\cite{MR3454090}.

The universal torsor method for Manin's conjecture was initiated by Salberger
\cite{MR1679841}, who gave a new proof of Manin's conjecture for split toric
varieties over $\Qd$, which are spherical. Moreover, estimating rational points
on a projective variety $X \subseteq \Pd^n$ by counting integral points on its
affine cone in $\Ad^{n+1}$, \eg by the circle method \cite{MR0150129}, can be
interpreted as an instance of the universal torsor method. However, all other
applications of the universal torsor method seem to concern non-spherical
varieties. In dimension $2$, there are many examples of smooth and singular del
Pezzo surfaces with a crepant resolution; see \cite{MR1909606, MR2874644,
  MR2332351, MR3100953}, for example. In higher dimension, only three cases are
known so far: Segre's singular cubic threefold \cite{MR2329549}, a singular
cubic fourfold \cite{MR3198752} and a singular biprojective cubic threefold
\cite{bbs16}; in all three cases, the singularities have a crepant resolution.
Hence all results proved by the universal torsor method are explained by Peyre's
relatively classical version of Manin's conjecture \cite[5.1]{MR2019019}.

The goal of our project is to start the investigation of Manin's conjecture for
spherical varieties by the universal torsor method. For this method, an explicit
description of the universal torsors is needed; this can be obtained from the
Cox rings of the underlying varieties (for details, see \cite{arXiv:1408.5358},
for example). Cox rings of spherical varieties were determined by Brion
\cite{bri07}. Also note that our results below are the first applications of the
universal torsor method to varieties without a crepant resolution, where the
more general conjectures of Batyrev and Tschinkel \cite{MR1679843} are relevant.

\subsection{A singular weighted cubic threefold and $(2 \times 2)$-determinants
  that are cubes}

One of the simplest spherical varieties that is neither horospherical
nor wonderful has the following nice and easy description: It is the
singular weighted cubic threefold
\begin{equation*}
  X_2 \coloneqq \Vd(ad-bc-z^3) \subseteq Y_2 \coloneqq \Pd_\Qd(1,2,1,2,1)
\end{equation*}
in the weighted projective space $Y_2$ with weighted homogeneous coordinates
$(a:b:c:d:z)$. It is closely related to the following Diophantine problem: How
often is the determinant of a $(2 \times 2)$-matrix a cube? The question of
representing a fixed number as a determinant over $\Zd$ is considered in
\cite{MR1230289}.

The action of the reductive group $\SL_2 \times \Gd_m$ defined by
\begin{equation*}
  \left(\begin{pmatrix}
  a' & b' \\
  c' & d'
\end{pmatrix},t\right)\cdot \left(\begin{pmatrix}
  a & b \\
  c & d
\end{pmatrix},z\right) \coloneqq \left(\begin{pmatrix}
  a' & b' \\
  c' & d'
\end{pmatrix}\cdot 
\begin{pmatrix}
  a & b \\
  c & d
\end{pmatrix}\cdot
\begin{pmatrix}
  t & 0 \\
  0 & t^{-1}
\end{pmatrix},z\right)
\end{equation*}
turns $X_2$ into a spherical variety. Its geometry can be analyzed by
the combinatorial theory of spherical varieties, which allows us to
determine its Picard number and its anticanonical divisor, for
example; we will do this in Section~\ref{sec:geometry}. For this
introduction, we emphasize a weighted-projective point of view; see
\cite{MR704986}.

Since $-K_{X_2} = \Om_{X_2}(4)$, we obtain an anticanonical height
\begin{equation*}
  H\colon X_2(\Qd) \to \Rd_{>0}
\end{equation*}
defined by
\begin{equation*}
  H(a:b:c:d:z) \coloneqq \frac{\max\{|a^4|,|b^2|,|c^4|,|d^2|,|z^4|\}}
  {\gcd(a^4,b^2,c^4,d^2,z^4)}
\end{equation*}
for $a,b,c,d,z \in \Zd$; note that in weighted projective space, we may not
assume that the coordinates are coprime. See also
Section~\ref{sec:parameterization}.

Blowing up its singular locus $\Vd(a,c,z) \cong \Pd^1_\Qd$ gives a crepant
resolution $\pi\colon \stX_2 \to X_2$ (in particular, $X_2$ has at worst
canonical singularities), with $\Pic(\stX_2)$ free of rank $2$. This means that
we are in the situation of Peyre's relatively classical version
\cite[5.1]{MR2019019} of Manin's conjecture. Our first main result (see
Theorem~\ref{thm:final_summations} for its proof) is compatible with this
prediction (see Section~\ref{sec:exp-x2}):

\begin{theorem}\label{thm:main-2_intro}
  We have
  \begin{equation*}
    N_{X_2,H}(B) = \cfr B\log B+O(B)\text{,}
  \end{equation*}
  where 
  \begin{equation*}
    \cfr = \frac{1}{8}\cdot\frac{1}{\zeta(2)\zeta(3)}\cdot
    \rleft(2 \iiiint_{|a|,|c|,|z|,|(ad-z^3)/c|,|d| \le 1} \frac{1}{|c|} \,\df a\,\df c\,\df d\,\df z\rright)
\end{equation*}
  is Peyre's constant.
\end{theorem}

\subsection{A family of spherical threefolds}

Our weighted cubic threefold $X_2 \subseteq \Pd_\Qd(1,2,1,2,1)$ can be
generalized as follows. For any positive integer $\n$, consider the weighted
hypersurface
\begin{equation*}
  X_\n \coloneqq \Vd(ad-bc-z^{\n+1}) \subseteq Y_\n \coloneqq \Pd_\Qd(1,\n,1,\n,1)
\end{equation*}
of degree $\n+1$ in the weighted projective space $Y_\n$ with weighted
homogeneous coordinates $(a:b:c:d:z)$. With an action of $\SL_2 \times \Gd_m$
that has the same description as above for $X_2$, each $X_\n$ is a spherical
threefold that is neither horospherical (see the beginning of
Section~\ref{sec:geometry}) nor wonderful (because it is not smooth).

Let $\n \ge 3$. By choosing sections of the very ample $\frac{\n}{\n+2}$-th
power of the $\Qd$-Cartier divisor $-K_{X_\n} = \Om_{X_\n}(\n+2)$, we obtain an
anticanonical height
\begin{equation*}
  H\colon X_\n(\Qd) \to \Rd_{>0}
\end{equation*}
defined by
\begin{equation*}
  H(a:b:c:d:z) = \left(\frac{\max\{|a^\n|,|b|,|c^\n|,|d|,|z^\n|\}}
    {\gcd(a^\n,b,c^\n,d,z^\n)}\right)^{\frac{\n+2}{\n}}
\end{equation*}
for $a,b,c,d,z \in \Zd$; see also Section~\ref{sec:parameterization}.

Naive heuristic considerations ignoring the denominator of the height function
(analogous to the ones in \cite[Heuristic principle]{MR2290498} and \cite[\S
5.1]{MR1679843}) lead to the expectation that $N_{X_\n,H}(B)$ might grow
linearly. However, in our second main result, we show (see
Theorem~\ref{thm:ge3-final} for its proof):

\begin{theorem}\label{thm:ge3_intro}
  Let $\n\ge 3$. We have
  \begin{equation*}
    N_{X_{\n,\reg},H}(B) = 
    \rleft(\sum_{x \in \Pd^2(\Qd) \setminus \Vd(a,c)} \cfr_x \rright) B^{\frac{2\n}{\n+2}} + O(B)\text{,}
  \end{equation*}
  where $X_{\n,\reg}$ denotes the smooth locus of $X_\n$. The values in the leading constant are
  \begin{equation*}
    \cfr_x = \frac{1}{2}\cdot\frac{1}{\zeta(2)}\cdot\omega_{\infty,x},
  \end{equation*}
  where (assuming that $a,c,z$ are coprime integral coordinates for $x$)
  \begin{align*}
    \omega_{\infty,x} = 
    \begin{dcases}
      \iint_{\substack{\subalign{
            |a^{\n}w|,|c^{\n}w|,|z^{\n}w| &\le 1\\
            |b|,|(bc+z^{\n+1}w)/a| &\le 1
          }}}  \frac{1}{|a|} \,\df b\,\df w & \text{for $a \ne 0$,}\\
      \iint_{\substack{\subalign{
            |a^{\n}w|,|c^{\n}w|,|z^{\n}w| &\le 1\\
            |(ad-z^{\n+1}w)/c|,|d| &\le 1
          }}}  \frac{1}{|c|} \,\df d\,\df w &\text{for $c \ne 0$.}
    \end{dcases}
  \end{align*}
\end{theorem}

We will see that $X_{n,\reg}$ is covered by rational curves, each of which
contains $\sim \cfr_x B^{2n/(n+2)}$ rational points of height at most $B$.
Therefore, we cannot obtain linear growth by removing a closed or thin subset.

Instead, we discuss in the next part of this introduction how our result is
explained by the predictions of Batyrev--Tschinkel \cite{MR1679843}; see
Section~\ref{sec:exp-ge-3} for more details. Note that the singular locus
$X_{\n,\sing}$ is a weakly accumulating subvariety, with
$N_{X_{\n,\sing},H}(B) \sim \frac{2}{\zeta(2)}B^{2\n/(\n+2)}$ (see
Remark~\ref{rem:singular_locus_counting}); we exclude it in
Theorem~\ref{thm:ge3_intro} to obtain a result that is compatible with
\cite{MR1679843}.

\subsection{The predictions of Batyrev--Tschinkel}

Let $X$ be a Fano variety over $\Qd$ with at worst log terminal singularities
and a Zariski dense set of rational points. Let $H\colon X(\Qd) \to \Rd_{>0}$ be
an anticanonical height function. Let $\pi\colon \stX \to X$ be a
desingularization and $L \coloneqq \pi^*(-K_X)$. By
\cite[Conjecture~C']{MR1032922} and \cite[3.6]{MR2019019}, we expect
\begin{equation*}
  N_{U,H}(B) \coloneqq \#\{x \in U(\Qd) : H(x) \le B\} \sim \cfr B^\afr (\log B)^{\bfr-1}
\end{equation*}
as $B \to \infty$, where $U$ is the complement of the closed (or thin) subset
consisting of the accumulating subvarieties,
$\afr\coloneqq \inf\{t\in \Rd : t \cdot L + K_{\tX}\text{ is effective}\}$ and
$\bfr$ is the codimension of the minimal face of the effective cone of $\stX$
containing $\smash{\afr \cdot L + K_{\tX}}$. Note that the effective cone of a
Fano variety with log terminal singularities is simplicial by
\cite[Corollary~1.3.2]{MR2601039}. If $X$ has at worst canonical singularities,
then $\smash{L + K_{\tX}}$ is effective, hence $\afr \le 1$. On the other hand,
for varieties with worse singularities, we may have $\afr > 1$, in which case
more than linear growth is expected.

A prediction for the leading constant $\cfr$ is given in \cite{MR1679843}. Here,
one considers the $\Lm$-primitive fibration (see
\cite[Definition~2.4.2]{MR1679843})
\begin{equation*}
  \phi\colon X \dashrightarrow P \coloneqq \Proj\rleft( \bigoplus_{\nu \ge 0}
  \Gamma\rleft(\stX, \rleft(\afr \cdot L + K_{\tX} \rright)^{\otimes \nu}\rright) \rright)\text{,}
\end{equation*}
and, for some restriction to open subsets $\phi\colon U \to V$, the constant
$\cfr$ is given by
\begin{equation*}
  \sum_{x \in V} \cfr_x,
\end{equation*}
where $\cfr_x$ is the expected constant in the asymptotic formula for the fiber
$\phi^{-1}(x)$. The sum should be taken over the fibers that contain a positive
proportion of the rational points (these are called $\Lm$-targets, see
\cite[Definition~3.2.4]{MR1679843}). If the divisor
$\smash{\afr \cdot L + K_{\tX}}$ is \emph{rigid}
(\cite[Definition~2.3.1]{MR1679843}, \eg if $X$ has a crepant resolution), then
the variety $P$ is a point.

Batyrev and Tschinkel make the following prediction in
\cite[Conjecture~3.5.1]{MR1679843}:

\begin{conjecture}\label{conj:BT}
  Let $\overline{H}$ be a height on $P$ relative to the line bundle
  $\Om_P(-1) \otimes \omega_P$. Then there exist positive constants
  $c_1,c_2$ and an open subset $V \subseteq P$ such that for every $x
  \in V$ we have
  \begin{equation*}
    c_1 \overline{H}(x) \le \cfr_x \le c_2 \overline{H}(x).
  \end{equation*}
\end{conjecture}

We apply the conjectures of \cite{MR1679843} to our family $X_\n$ of spherical
varieties; see Section~\ref{sec:geometry} for their geometry. Blowing up the
singular locus $\Vd(a,c,z) \cong \Pd^1_\Qd$ gives a desingularization
$\pi\colon \stX_\n \to X_\n$, and we will see that we have
\begin{align*}
  \afr=\frac{2\n}{\n+2} &&\text{and}&&
  \bfr=
  \begin{cases}
     2, & \text{for $\n=2$,}\\
     1, & \text{for $\n\ge 3$.}
  \end{cases}
\end{align*}

For $\n \ge 3$, the singularities of $X_\n$ are not canonical, but log terminal.
The divisor $\smash{\afr\cdot \pi^*(-K_{X_\n})+K_{\tX_\n}}$ is not rigid, and
the $\Lm$-primitive fibration turns out to be a map
$\phi\colon X_\n \dashrightarrow P_\n \cong \Pd^2_\Qd$ with
$\smash{\Om_{P_\n}(1) \cong \Om_{\Pd^2_\Qd}(\n-2)}$ such that the constants
$\cfr_x$ appearing in Theorem~\ref{thm:ge3_intro} are Peyre's constant for the
fibers $\fibphix$.

In the proofs of Theorems~\ref{thm:main-2_intro} and \ref{thm:ge3_intro}, we
work with universal torsors over a further blow-up
$\shX_\n \to \stX_\n \to X_\n$ because this leads to more convenient coprimality
conditions in the associated counting problem (see Remark~\ref{rem:blt}). This
seems surprising to us because proofs of cases of Manin's conjecture for
singular del Pezzo surfaces usually use universal torsors of their minimal
desingularizations.

It turns out that Conjecture~\ref{conj:BT} of Batyrev--Tschinkel is true for
$X_\n$ (see Theorem~\ref{thm:height_constant}):

\begin{theorem}\label{thm:height_constant_intro}
  Let $\overline{H}\colon \Pd^2(\Qd) \to \Rd_{>0}$ be a height relative to
  \begin{align*}
    \Om_{\Pd^2_{\Qd}}(-\n-1) \cong \Om_{P_\n}(-1) \otimes \omega_{P_\n}\text{.}
  \end{align*}
  There exist positive constants $c_1, c_2$ such that for every
  $x \in \Pd^2(\Qd) \setminus \Vd(a,c)$ we have
  \begin{align*}
    c_1 \overline{H}(x) \le \cfr_x \le c_2 \overline{H}(x)\text{.}
  \end{align*}
\end{theorem}

This implies that the sum over the constants $\cfr_x$ in
Theorem~\ref{thm:ge3_intro} converges.

\subsection{A second family of spherical threefolds}

Since the varieties $X_\n$ considered above are equivariant compactifications of
$\Gd_a^3$, Manin's conjecture is already known for them by \cite{MR1906155} (for
heights corresponding to smooth adelic metrics; note that we work with a height
corresponding to an adelic metric that is not smooth). To illustrate that our
approach can also be applied to spherical varieties without such a structure, we
consider a family of varieties $X'_\n$ for $\n \ge 2$ that do not belong to any
of the classes of varieties for which Manin's conjecture is known.

A comparison of the geometric description, the shape of the main results and
their proofs for the family $X_\n$ with the family $X'_\n$ will reveal many
similarities, but also several additional complications for $X'_\n$. In
particular, we will see that Conjecture~\ref{conj:BT} fails for $X'_\n$. Hence
the family $X_\n$ can be regarded as a warm-up for the family $X'_\n$.

Fix an integer $\n \ge 2$. Consider the weighted projective space $Y_{\n-1}$
with Cox coordinates $(a:b:c:d:y)$ and the toric modification
$Y'_\n \to Y_{\n-1}$ obtained by first blowing up the singular locus of
$Y_{\n-1}$, then blowing up the two torus invariant curves in the resulting
exceptional divisor, and finally contracting the exceptional divisor from the
first step. With Cox coordinates $(a:b:c:d:y:z:t)$, where $z$ corresponds to the
torus invariant curve in $Y'_{\n-1}$ contained in $\Vd(y)$ and $t$ to the other
one, we consider the hypersurface
\begin{equation*}
  X'_\n \coloneqq \Vd(ad-bc-y^\n z^{\n+1}) \subseteq Y'_\n\text{.}
\end{equation*}
Equipped with a suitable action of the reductive group
$\SL_2 \times \Gd_m$, it is a singular spherical threefold that is
neither horospherical (see the beginning of
Section~\ref{sec:geometry}) nor wonderful (because it is not smooth);
moreover it is not isomorphic to an equivariant compactification of
$\Gd_a^3$ since its effective cone can be shown not to be simplicial.

In Section~\ref{sec:geometry}, we will construct a desingularization
$\pi\colon \stX'_\n \to X'_\n$, and we will see that we have
\begin{align*}
  \afr=\frac{2\n+2}{\n+3} &&\text{and}&&
  \bfr=1\text{.}
\end{align*}
In Section~\ref{sec:parameterization}, we will construct an anticanonical height
\begin{equation*}
  H'\colon X'_\n \to \Rd_{>0}
\end{equation*}
by choosing sections of a very ample power of the $\Qd$-Cartier divisor
$\pi^*(-K_{X'_\n})$ on $\stX'_\n$.

The singularities of $X'_\n$ are log terminal, and the divisor
$\smash{\afr\cdot \pi^*(-K_{X'_\n})+K_{\tX'_\n}}$ is not rigid. We will find the
$\Lm$-primitive fibration
$\phi'\colon X'_\n \dashrightarrow P'_\n \cong \Pd^2_\Qd$, where we denote the
homogeneous coordinates of $\smash{\Pd^2_\Qd}$ by $(\ca:\cc:\cy)$. Again, our
main result (see Theorem~\ref{thm:bpc}) is compatible with the predictions of
\cite{MR1679843} (see Section~\ref{sec:exp-b-ge-2}).

\begin{theorem}\label{thm:b_intro}
  Let $\n \ge 2$ and $U' \coloneqq X'_\n \setminus \Vd(yzt)$. For every $\epsilon > 0$, we have
  \begin{equation*}
    N_{U',H'}(B) = 
    \rleft(\sum_{x \in \Pd^2(\Qd) \setminus (\Vd(\ca,\cc)\cup \Vd(\cy))} \cfr_x \rright) B^{\frac{2\n+2}{\n+3}} 
    + O_\epsilon(B^{1+\epsilon})\text{,}
  \end{equation*}
  where each summand $\cfr_x$ in the leading constant is Peyre's constant for
  the rational fiber $\fibphibx$. Its value is
  \begin{equation*}
    \cfr_x = \frac{1}{2}\cdot \rleft(\prod_{p\text{ prime}} \rleft(1-\frac{1}{p}\rright)\omega_{p,x}\rright)\cdot\omega_{\infty,x},
  \end{equation*}
  with (assuming that $\ca,\cc,\cy$ are coprime integral coordinates for $x$ and
  $\efr \coloneqq \tfrac{-n+1}{\n+3}$)
  \begin{align*}
    \omega_{p,x} = \bigg(\rleft(1-\frac{1}{p}\rright)\cdot \frac{1-(p^\efr)^{\nu_p(\cy)+1}}{1-p^\efr}+
    \frac{1}{p}+\frac{(p^\efr)^{\nu_p(\cy)}}{p}\bigg)
    \cdot (p^{\efr+1})^{\min(\nu_p(\ca), \nu_p(\cc))},
  \end{align*}
  and
  \begin{align*}
    \omega_{\infty,x} =
    \begin{dcases}
      \iint_{\max|\Ms'_\n(\ca,b,\cc,(b\cc+\cy^\n w)/\ca,\cy,1,1,w)| \le 1} \frac{1}{|\ca|} \,\df b\,\df w & \text{for $\ca \ne 0$,}\\
      \iint_{\max|\Ms'_\n(\ca,(\ca d -\cy^\n w)/\cc, \cc, d,\cy,1,1,w)|
        \le 1} \frac{1}{|\cc|} \,\df d\,\df w & \text{for $\cc \ne
        0$,}
    \end{dcases}
  \end{align*}
  where $\Ms'_\n(\dots)$ denotes the set of $13$ monomials
  from Remark~\ref{rem:monomials}.
\end{theorem}
In particular, the expressions for the $p$-adic densities $\omega_{p,x}$ are
apparently much more complicated than in previous applications of the universal
torsor method for Manin's conjecture. Also note that $\omega_{p,x}$ depends on
the base point $x$, while the $p$-adic densities in Theorem~\ref{thm:ge3_intro}
are independent of $x$.

Finally, Conjecture~\ref{conj:BT} of Batyrev--Tschinkel is not true for
$X'_\n$. In fact, even a weaker \enquote{up to $\epsilon$}-version of this
conjecture fails (see Theorem~\ref{thm:height_constant_bar}; roughly, the
reason is that $\gcd(\ca,\cc)^{\efr+1}$ appears in the product of the $p$-adic
densities in $\cfr_x$):

\begin{theorem}\label{thm:height_constant_intro_b}
  Let $\overline{H}\colon \Pd^2(\Qd) \to \Rd_{>0}$ be a height relative to an
  arbitrary line bundle. Then there are $\epsilon > 0$ such that there does
  not exist an open subset $V \subseteq \Pd^2(\Qd)$ with positive constants
  $c_1, c_2$ such that for every $x \in V$ we have
  \begin{align*}
    c_1 \overline{H}(x)^{1-\epsilon} \le \cfr_x \le c_2 \overline{H}(x)^{1+\epsilon}\text{.}
  \end{align*}
\end{theorem}

Nevertheless, we can show that the sum over the constants $\cfr_x$ in
Theorem~\ref{thm:b_intro} converges (see Proposition~\ref{prop:cxb-finite}).

See \cite[\S 4.2]{arXiv:1707.03231} for a second example where
Conjecture~\ref{conj:BT} fails; in that case of a certain conic bundle
over $\Pd^1$, however, the upper bound of the conjecture holds
\enquote{up to $\epsilon$}. For an investigation of the behavior of
Peyre's constant for families of diagonal quartic threefolds, see
\cite[Theorem~1.6]{MR2333743}.

\subsection*{Acknowledgements}

The authors are grateful to Daniel Loughran and the referee for several useful comments.

\section{Two families of spherical hypersurfaces in toric
  varieties}\label{sec:geometry}

Let $G$ be a connected reductive group over $\Kc$, and let
$B \subseteq G$ be a Borel subgroup. A normal $G$-variety $X$ over
$\Kc$ is called \emph{spherical} if it contains a dense $B$-orbit.
Over an algebraically closed field of characteristic $0$ (such as
$\Kc$), there is a complete combinatorial description of spherical
varieties. First, spherical homogeneous spaces are described by a
program initiated by Luna \cite{lun01}, which has been recently
completed \cite{bp15b,cf2,los09a}. Then, given a spherical homogeneous
space $G/H$, the Luna--Vust theory \cite{lv83,kno91} describes all
\emph{spherical embeddings}, \ie $G$-equivariant open embeddings
$G/H \hookrightarrow X$ into a normal irreducible $G$-variety $X$, in
terms of \emph{colored fans}, which generalize the fans of toric
varieties. For further details, we refer to the general references
\cite{bl11,per14,tim11}.

For $G \coloneqq \SL_2$ the spherical $G$-varieties are at most $2$-dimensional,
where each complete one is isomorphic to $\Pd^1 \times \Pd^1$, $\Pd^2$, or the
blow-up of $\Pd^2$ in one point. The next possible step is to consider
$G \coloneqq \SL_2 \times \Gd_m$.
Let $\varepsilon \colon \Gd_m \to \Gd_m$ be a primitive
character and consider the spherical subgroup
\begin{align*}
  H \coloneqq \{(\lambda, \varepsilon(\lambda)) : \lambda \in T_{\SL_2}\} \subseteq G\text{,}
\end{align*}
where $T_{\SL_2}$ denotes a maximal torus in $\SL_2$. Then
$T \coloneqq T_{\SL_2} \times \Gd_m$ is maximal torus in $G$. Let
$B \subseteq G$ be a Borel subgroup containing $T$. Let
$\alpha \in \Xf(T) = \Xf(B)$ be the unique simple root corresponding
to these choices.

We can now briefly introduce the central combinatorial objects
associated to $G/H$ by the Luna-Vust theory. The \emph{weight lattice}
$\Mm \subseteq \Xf(B)$, \ie the lattice of weights of
$B$-semi-invariants (or $B$-eigenvectors) occurring in $\Qbar(G/H)$,
has basis
$(\frac{1}{2}\alpha + \varepsilon, \frac{1}{2}\alpha - \varepsilon)$.
The \emph{set of colors} $\Dm$, \ie the set of $B$-invariant prime
divisors in $G/H$, contains two elements, which we denote by $D'$ and
$D''$. The set $\Dm$ is equipped with the map
$\rho \colon \Dm \to \Nm \coloneqq \Hom(\Mm, \Zd)$ defined by
$\langle \rho(D), \chi \rangle \coloneqq \nu_D(f_\chi)$ where $\nu_D$
is the valuation on $\Qbar(G/H)$ which is induced by the prime divisor
$D$ and $f_\chi \in \Qbar(G/H)$ is a $B$-semi-invariant of weight
$\chi \in \Mm$ (which is defined up to a constant factor because of
the open $B$-orbit). We can choose $D'$ and $D''$ such that
$(\rho(D'), \rho(D''))$ is the dual basis to
$(\frac{1}{2}\alpha + \varepsilon, \frac{1}{2}\alpha - \varepsilon)$
of $\Nm$. Finally, the \emph{valuation cone}
$\Vm \subseteq \Nm_{\Qd} \coloneqq \Hom(\Mm, \Qd)$, which can be
identified with the $\Qd$-valued $G$-invariant discrete valuations on
$\Qbar(G/H)$, is given by
$\Vm = \{v \in \Nm_\Qd : \langle v, \alpha \rangle \le 0\}$. Spherical
varieties with $\Vm = \Nm_\Qd$ are called \emph{horospherical}. In
particular, because we have $\Vm \ne \Nm_\Qd$, no embedding of our
example $G/H$ is horospherical. The situation inside the vector space
$\Nm_\Qd$ is illustrated in the following picture.

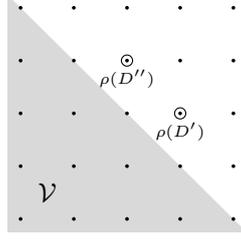
\begin{align*}
  \begin{tikzpicture}[scale=0.7]
    \clip (-2.24, -2.24) -- (2.24, -2.24) -- (2.24, 2.24) -- (-2.24, 2.24) -- cycle;
    \fill[color=gray!30] (-3, 3) -- (3, -3) -- (-3, -3) -- cycle;
    \foreach \x in {-3,...,3} \foreach \y in {-3,...,3} \fill (\x, \y) circle (1pt);
    \draw (1, 0) circle (3pt);
    \draw (0, 1) circle (3pt);
    \path (1, 0) node[anchor=north] {{\tiny{$\rho(D')$}}};
    \path (0, 1) node[anchor=north] {{\tiny{$\rho(D'')$}}};
    \path (-1.5, -1.5) node {{$\Vm$}};
  \end{tikzpicture}
\end{align*}

A spherical embedding $G/H \hookrightarrow X$ is now described by a
\emph{colored fan}, which is a set of \emph{colored cones}, which are
pairs $(\Cm, \Fm)$ where $\Cm$ is a polyhedral cone in $\Nm_{\Qd}$ and
$\Fm$ is a subset of $\Dm$, and where moreover certain properties and
compatibility conditions are satisfied. Similarly to the case of toric
varieties, the colored cones are in bijection with the $G$-orbits in
$X$. The colored cones corresponding to $G$-orbits of codimension $1$
are easier to describe: they have the form $(\rho, \emptyset)$ where
$\rho$ is a ray in $\Vm$, which means that we have $\rho = \cone(u)$
for a uniquely determined primitive element $u \in \Vm \cap \Nm$.

Now let $G/H \hookrightarrow X$ be a spherical embedding, and let
$u_1, \dots, u_n \in \Vm \cap \Nm$ be the primitive elements
corresponding to (the open orbits in) the $G$-invariant prime divisors
$D_1, \dots, D_n$ in $X$. According to
\cite[Proposition~4.1.1]{bri07}, the divisor class group $\Cl(X)$ is
generated by divisor classes $[D_1], \dots, [D_n]$ and the divisor
classes of the colors $\Dm$, and the relations can be computed from
the relative position of the $u_1, \dots, u_n \in \Nm$ similarly to
the toric case. Moreover, the Cox ring $\Rm(X)$ of $X$ can be obtained
explicitly using \cite[Theorem~4.3.2]{bri07} or
\cite[Theorem~3.6]{gag14}:

\begin{prop}
  \label{prop:cox1}
  Let $r_i \coloneqq -\langle u_i, \alpha \rangle$. Then we have
  \begin{align*}
    \Rm(X) = \Kc[a, b, c, d, z_1, \dots, z_n]/\langle ad-bc-z_1^{r_1}\cdots z_n^{r_n}\rangle
  \end{align*}
  with $\deg(a) = \deg(c) = [D']$, $\deg(b) = \deg(d) = [D'']$, and
  $\deg(z_i) = [D_{i}]$.
\end{prop}

For every $\n \ge 2$, we consider the spherical embedding
$G/H \hookrightarrow X_\n$ with exactly one $G$-invariant prime divisor
corresponding to the primitive element
\begin{align*}
  u_z \coloneqq -\rho(D')-\n\rho(D'') \in \Vm \cap \Nm\text{.}
\end{align*} 
It can be shown that $X_\n$ is isomorphic to an equivariant compactification of
$\Gd_a^3$.

We therefore also consider the spherical embedding
$G/H \hookrightarrow X_\n'$ with two additional $G$-invariant prime
divisors corresponding to the primitive elements
\begin{align*}
  u_y &\coloneqq -\rho(D')-(\n-1)\rho(D'') \in \Vm \cap \Nm\text{,}\\
  u_t &\coloneqq \rho(D')-\rho(D'') \in \Vm \cap \Nm\text{.}
\end{align*}
It can be shown that the effective cone of $X'_\n$ is not simplicial, hence
$X'_\n$ is not isomorphic to an equivariant compactification of $\Gd_a^3$. The
colored fans of $X_2$ and $X_2'$ are illustrated in the following pictures.
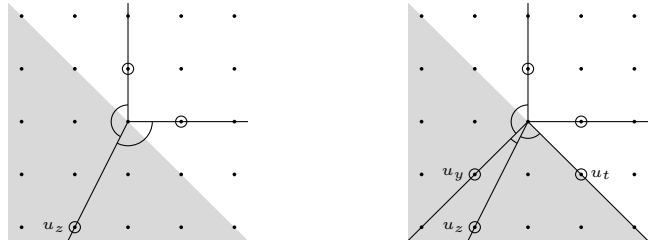
\begin{align*}
  \begin{tikzpicture}[scale=0.7]
    \clip (-2.24, -2.24) -- (2.24, -2.24) -- (2.24, 2.24) -- (-2.24, 2.24) -- cycle;
    \fill[color=gray!30] (-3, 3) -- (3, -3) -- (-3, -3) -- cycle;
    \foreach \x in {-3,...,3} \foreach \y in {-3,...,3} \fill (\x, \y) circle (1pt);
    \draw (1, 0) circle (3pt);
    \draw (0, 1) circle (3pt);
    \draw (-1, -2) circle (3pt);
    \draw (0,0) -- (3,0);
    \draw (0,0) -- (0,3);
    \draw (0,0) -- (-2, -4);
    \path (-1, -2) node[anchor=east] {{\tiny{$u_z$}}};
    \begin{scope}
      \clip (0,0) -- (0,1) -- (-1,1) -- (-1, -2) -- cycle; \draw (0,0) circle (9pt);
    \end{scope}
    \begin{scope}
      \clip (0,0) -- (-1, -2) -- (1, -2) -- (1, 0) -- cycle; \draw (0,0) circle (13pt);
    \end{scope}
  \end{tikzpicture}
  &&
  \begin{tikzpicture}[scale=0.7]
    \clip (-2.24, -2.24) -- (2.24, -2.24) -- (2.24, 2.24) -- (-2.24, 2.24) -- cycle;
    \fill[color=gray!30] (-3, 3) -- (3, -3) -- (-3, -3) -- cycle;
    \foreach \x in {-3,...,3} \foreach \y in {-3,...,3} \fill (\x, \y) circle (1pt);
    \draw (1, 0) circle (3pt);
    \draw (0, 1) circle (3pt);
    \draw (-1, -2) circle (3pt);
    \draw (-1, -1) circle (3pt);
    \draw (1, -1) circle (3pt);
    \draw (0,0) -- (3,0);
    \draw (0,0) -- (0,3);
    \draw (0,0) -- (-2, -4);
    \draw (0,0) -- (-4, -4);
    \draw (0,0) -- (4, -4);
    \path (-1, -2) node[anchor=east] {{\tiny{$u_z$}}};
    \path (-1, -1) node[anchor=east] {{\tiny{$u_y$}}};
    \path (1, -1) node[anchor=west] {{\tiny{$u_t$}}};
    \begin{scope}
      \clip (0,0) -- (0,1) -- (-1,1) -- (-1, -1) -- cycle; \draw (0,0) circle (9pt);
    \end{scope}
    \begin{scope}
      \clip (0,0) -- (-1,-2) -- (-1,-1) -- cycle; \draw (0,0) circle (13pt);
    \end{scope}
    \begin{scope}
      \clip (0,0) -- (-1, -2) -- (1, -2) -- (1, -1) -- cycle; \draw (0,0) circle (9pt);
    \end{scope}
  \end{tikzpicture}
\end{align*}

Using \cite{arXiv:1408.5358}, we consider $X_\n$ and $X'_\n$ as varieties over
$\Qd$. According to Proposition~\ref{prop:cox1}, we have
\begin{align*}
  \Rm(X_\n) = \Qd[a,b,c,d,z]/\langle ad-bc-z^{\n+1} \rangle
\end{align*}
with $\Cl(X_\n) \cong \Z$ where $\deg(a) = \deg(c) = \deg(z) = 1$ and
$\deg(b) = \deg(d) = \n$. Moreover, the graded ring $\Qd[a,b,c,d,z]$, where we
ignore the relation, is identified as the Cox ring of the weighted projective
space \smash{$Y_\n \coloneqq \Pd_{\Qd}(1,\n,1,\n,1)$}. It follows that $X_\n$ is
a hypersurface in $Y_\n$ defined by $ab-cd-z^{\n+1} = 0$.

Similarly, we have
\begin{align*}
  \Rm(X_\n') = \Qd[a,b,c,d,y,z,t]/\langle ad-bc-y^{\n}z^{\n+1} \rangle
\end{align*}
with $\Cl(X_\n') \cong \Z^3$ where $\deg(a) = \deg(c) = (1, 1, -1)$,
$\deg(b) = \deg(d) = (\n, \n-1, 1)$, $\deg(z) = (1,0,0)$, $\deg(y) = (0,1,0)$,
and $\deg(t) = (0,0,1)$. Again, the variety $X_\n'$ is a hypersurface in a toric
variety $Y_\n'$ with graded Cox ring $\Qd[a,b,c,d,y,z,t]$.

According to \cite[4.1 and 4.2]{bri97} or \cite[Proposition~3.3.3.2]{adhl15}, we
have the anticanonical divisor classes $-K_{X_\n} = \n+2$ and
$-K_{X_\n'} = (\n+2,\n+1,1)$. Moreover, according to \cite[Theorem~1.9]{gh15} or
\cite[3.3.2.9]{adhl15}, the varieties $X_\n$ and $X_\n'$ are Fano for every
$\n \ge 2$, and the variety $X_2$ is Gorenstein.

The singular loci are $X_{\n,\sing} = X \cap \Vd(a,c,z)$ and
$X'_{\n,\sing} = X' \cap \Vd(z,t)$. We construct desingularizations
$\pi\colon \stX_\n \to X_\n$ and $\pi'\colon \stX'_\n \to X'_\n$ by subdividing
their colored fans. We add a $G$-invariant prime divisor corresponding to the
primitive element $u_w \coloneqq -\rho(D'') \in \Vm \cap \Nm$. The resulting
colored fans of the spherical varieties $\stX_2$ and $\stX_2'$ are illustrated
in the following pictures.

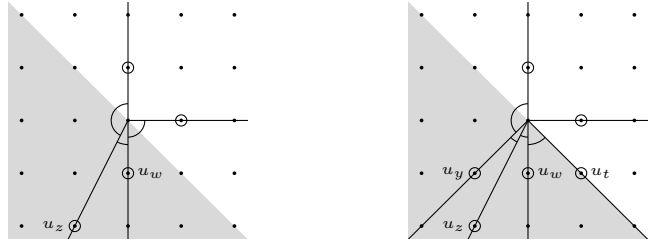
\begin{align*}
  \begin{tikzpicture}[scale=0.7]
    \clip (-2.24, -2.24) -- (2.24, -2.24) -- (2.24, 2.24) -- (-2.24, 2.24) -- cycle;
    \fill[color=gray!30] (-3, 3) -- (3, -3) -- (-3, -3) -- cycle;
    \foreach \x in {-3,...,3} \foreach \y in {-3,...,3} \fill (\x, \y) circle (1pt);
    \draw (1, 0) circle (3pt);
    \draw (0, 1) circle (3pt);
    \draw (-1, -2) circle (3pt);
    \draw (0, -1) circle (3pt);
    \draw (0,0) -- (3,0);
    \draw (0,0) -- (0,3);
    \draw (0,0) -- (-2, -4);
    \draw (0,0) -- (0, -4);
    \path (-1, -2) node[anchor=east] {{\tiny{$u_z$}}};
    \path (0, -1) node[anchor=west] {{\tiny{$u_w$}}};
    \begin{scope}
      \clip (0,0) -- (0,1) -- (-1,1) -- (-1, -2) -- cycle; \draw (0,0) circle (9pt);
    \end{scope}
    \begin{scope}
      \clip (0,0) -- (-1, -2) -- (0, -2) -- cycle; \draw (0,0) circle (13pt);
    \end{scope}
    \begin{scope}
      \clip (0,0) -- (0, -1) -- (1, -2) -- (1, 0) -- cycle; \draw (0,0) circle (9pt);
    \end{scope}
  \end{tikzpicture}
&&
  \begin{tikzpicture}[scale=0.7]
    \clip (-2.24, -2.24) -- (2.24, -2.24) -- (2.24, 2.24) -- (-2.24, 2.24) -- cycle;
    \fill[color=gray!30] (-3, 3) -- (3, -3) -- (-3, -3) -- cycle;
    \foreach \x in {-3,...,3} \foreach \y in {-3,...,3} \fill (\x, \y) circle (1pt);
    \draw (1, 0) circle (3pt);
    \draw (0, 1) circle (3pt);
    \draw (-1, -2) circle (3pt);
    \draw (-1, -1) circle (3pt);
    \draw (1, -1) circle (3pt);
    \draw (0, -1) circle (3pt);
    \draw (0,0) -- (3,0);
    \draw (0,0) -- (0,3);
    \draw (0,0) -- (-2, -4);
    \draw (0,0) -- (-4, -4);
    \draw (0,0) -- (4, -4);
    \draw (0,0) -- (0, -4);
    \path (-1, -2) node[anchor=east] {{\tiny{$u_z$}}};
    \path (-1, -1) node[anchor=east] {{\tiny{$u_y$}}};
    \path (1, -1) node[anchor=west] {{\tiny{$u_t$}}};
    \path (0, -1) node[anchor=west] {{\tiny{$u_w$}}};
    \begin{scope}
      \clip (0,0) -- (0,1) -- (-1,1) -- (-1, -1) -- cycle; \draw (0,0) circle (9pt);
    \end{scope}
    \begin{scope}
      \clip (0,0) -- (-1,-2) -- (-1,-1) -- cycle; \draw (0,0) circle (13pt);
    \end{scope}
    \begin{scope}
      \clip (0,0) -- (-1, -2) -- (1, -2) -- (0, -1) -- cycle; \draw (0,0) circle (9pt);
    \end{scope}
    \begin{scope}
      \clip (0,0) -- (0, -1) -- (1, -1) -- cycle; \draw (0,0) circle (13pt);
    \end{scope}
  \end{tikzpicture}
\end{align*}
According to Proposition~\ref{prop:cox1}, we have
\begin{align*}
  \Rm(\tX_\n) = \Qd[a,b,c,d,z,w]/\langle ad-bc-z^{\n+1}w \rangle
\end{align*}
with $\Pic(\stX_\n) \cong \Cl(\stX_\n) \cong \Z^2$ where
$\deg(a) = \deg(c) = \deg(z) = (1,0)$, $\deg(b) = \deg(d) = (\n, 1)$, and
$\deg(w) = (0,1)$. Moreover, we have
\begin{align*}
  \Rm(\tX_\n') = \Qd[a,b,c,d,y,z,t,w]/\langle ad-bc-y^{\n}z^{\n+1}w \rangle
\end{align*}
with $\Pic(\stX_\n') \cong \Cl(\stX_\n') \cong \Z^4$ where
$\deg(a) = \deg(c) = (1,1,-1,0)$, $\deg(b) = \deg(d) = (\n, \n-1, 1, 1)$,
$\deg(z) = (1,0,0,0)$, $\deg(y) = (0,1,0,0)$, $\deg(t) = (0,0,1,0)$, and
$\deg(w) = (0,0,0,1)$.

In order to obtain explicit descriptions of $\stX_\n$ and $\stX'_\n$, we use
\cite[Theorem~2.2.2.2, Proposition~3.3.2.9, and Construction~3.2.1.3]{adhl15}
and \cite{arXiv:1408.5358}, according to which the quasi-affine varieties
\begin{align*}
  \Tm_\n &\coloneqq \Spec(\Rm(\tX_\n)) \setminus (\Vd(a,c,z) \cup \Vd(b,d,w))\text{,}\\
  \Tm_\n' &\coloneqq \Spec (\Rm(\tX_\n')) \setminus (\Vd(a,c) \cup \Vd(b,d,z) \cup \Vd(b,d,w) \cup \Vd(y, w) \cup \Vd(y, t) \cup \Vd(z, t))
\end{align*}
are universal torsors $\Tm_\n \to \stX_\n$ and $\Tm_\n' \to \stX_\n'$ with
respect to the natural actions of the tori
\begin{equation*}
  \Spec(\Qd{}[\Pic(\stX_\n)]) \cong \Gd_m^2 \text{ and }
  \Spec(\Qd{}[\Pic(\stX'_\n)]) \cong \Gd_m^4
\end{equation*}
respectively.

According to \cite[4.1 and 4.2]{bri97} or \cite[Proposition~3.3.3.2]{adhl15}, we
have
\begin{align*}
  -K_{\tX_\n} &= (\n+2,2)\text{,} & \pi^*(-K_{X_\n}) &= \rleft(\n+2,\tfrac{\n+2}{\n}\rright)\text{,}\\
  -K_{\tX'_\n} &= (\n+2,\n+1,1,2)\text{,} & \pi^*(-K_{X'_\n}) &= \rleft(\n+2,\n+1,1,\tfrac{\n+3}{\n+1}\rright)\text{.}
\end{align*}
In particular, the resolution $\pi\colon \stX_\n \to X_\n$ is crepant and $X_\n$
has at worst canonical singularities if and only if $\n=2$ (see, for instance,
\cite{ab04}).

\section{Parameterization of rational points via universal torsors}
\label{sec:parameterization}

Using the universal torsors $\Tm_\n$ and $\Tm_\n'$ from
Section~\ref{sec:geometry}, we parameterize the rational points on $X_\n$ and
$X_\n'$, respectively.

Consider the line bundles
\begin{align*}
  L &\coloneqq \tfrac{\n}{\n+2} \cdot \pi^*(-K_{X_\n}) = (\n,1)\text{,} \\
  L' &\coloneqq (\n+1)\cdot\pi^*(-K_{X'_\n}) = (\n^2+3\n+2, \n^2+2\n+1, \n+1, \n+3)\text{.}
\end{align*}
We define
\begin{align*}
  \Ms_{\n}(a,b,c,d,z,w) &\coloneqq \{\text{monomials in $\Rm(\stY_\n)$ of degree $L$ restricted to $\stX_\n$}\}\text{,}\\
  \Ms'_{\n}(a,b,c,d,y,z,t,w) &\coloneqq \{\text{monomials in $\Rm(\stY'_\n)$ of degree $L'$ restricted to $\stX'_\n$}\}\text{.}
\end{align*} 
Then we have
\begin{align*}
  H(\pi(a:b:c:d:z:w)) &\coloneqq \left(\frac{\max|\Ms_{\n}(a,b,c,d,z,w)|}{\gcd \Ms_{\n}(a,b,c,d,z,w)}\right)^{(\n+2)/\n}\text{,}\\
  H'(\pi(a:b:c:d:y:z:t:w)) &\coloneqq \left(\frac{\max|\Ms'_{\n}(a,b,c,d,y,z,t,w)|}{\gcd \Ms'_{\n}(a,b,c,d,y,z,t,w)}\right)^{1/(\n+1)}
\end{align*}
for anticanonical heights $H$ and $X_\n$ and $H'$ on $X'_\n$.

We are now going to state the counting problem for $X_\n$. We consider the open
subset
\begin{align*}
  U \coloneqq \stX_\n \setminus \Vd(w) = X_\n \setminus \Vd(a,c,z)\text{.}
\end{align*}

\begin{prop}
  \label{prop:41}
  There is a natural $4$-to-$1$ correspondence between
  \begin{align*}
    \Um \coloneqq \left\{(a,b,c,d,z,w) \in \Zd^6 : 
    \begin{aligned}
      &w \ne 0;\ ad-bc-z^{\n+1}w=0\\
      &\gcd(a,c,z)=\gcd(b,d,w)=1
    \end{aligned}
        \right\}
  \end{align*}
  and the set $U(\Qd)$. Moreover, for $(a,b,c,d,z,w) \in \Um$, we have
  \begin{equation*}
    H(\pi(a:b:c:d:z:w)) = \max|\Ms_{\n}(a,b,c,d,z,w)|^{(\n+2)/\n}\text{.}
  \end{equation*}
\end{prop}
\begin{proof}
  The toric variety $\stY_\n$ comes from a regular fan. According to
  \cite[Section~8]{MR1679841}, we may construct a toric scheme $\stYf_\n$ over
  $\Spec(\Zd)$, together with a map
  \begin{align*}
    \mathfrak{W}_\n \coloneqq \Spec(\Z[a,b,c,d,z,w]) \setminus (\Vd(a,c,z)\cup \Vd(b,d,w)) \to \stYf_\n\text{,}
  \end{align*}
  which is a model for the universal torsor $\mathcal{W}_\n \to \stY_\n$, obtain
  a $4$-to-$1$ quotient
  \begin{align*}\mathfrak{W}_\n(\Zd) \to \stYf_\n(\Zd) = \stY_\n(\Qd)
  \end{align*}
  for the $\Gd_m^2(\Zd) \cong \{\pm 1\}^2$-action as well as
  the claim on the height function. As we 
  have
  \begin{align*}
    \mathfrak{W}_\n(\Zd) = \left\{(a,b,c,d,z,w) \in \Zd^6 :
    \gcd(a,c,z)=\gcd(b,d,w)=1
    \right\}\text{,}
  \end{align*}
  the result follows after restricting to the equation $ad-bc-z^{\n+1}w=0$.

  Alternatively, the claims can easily be verified by elementary manipulations
  of the defining equation.
\end{proof}

\begin{cor}\label{cor:countingproblem_i}
  We have that $N_{U,H}(B^{(\n+2)/\n})$ is equal to
  \begin{align*}
    \frac{1}{4} \#\left\{(a,b,c,d,z,w) \in \Zd^6 : 
    \begin{aligned}
      &w \ne 0;\ ad-bc-z^{\n+1}w=0\\
      &\gcd(a,c,z)=\gcd(b,d,w)=1\\
      &\max|\Ms_\n(a,b,c,d,z,w)| \le B
    \end{aligned}
        \right\}.
  \end{align*}
\end{cor}

As the Diophantine equation $ad-bc = z^{\n+1}w$ is easier to solve for $d$ or
$b$ under the additional condition $\gcd(a,c) = 1$, we will use the following
counting problem, which introduces an additional variable.

\begin{cor}\label{cor:countingproblem}
  We have that $N_{U,H}(B^{(\n+2)/\n})$
  is equal to
  \begin{align*}
    \frac{1}{8} \#\left\{(a,b,c,d,z,w,t) \in \Zd^7 :
    \begin{aligned}
      &wt \ne 0;\ ad-bc-z^{\n+1}w=0\\
      &\gcd(a,c)=\gcd(b,d,w)=\gcd(z,t)=1\\
      &|a^{\n}wt^{\n+1}|,|c^{\n}wt^{\n+1}|,|z^{\n}wt|,|b|,|d| \le B
    \end{aligned}
        \right\}.
  \end{align*}
\end{cor}

\begin{remark}
  \label{rem:blt}
  Corollary~\ref{cor:countingproblem} can be interpreted as a version of
  Corollary~\ref{cor:countingproblem_i}, where instead of the desingularization
  $\stX_\n \to X_\n$, we use a further blow-up
  \begin{equation*}
    \hX_\n \to \stX_\n \to X_\n.
  \end{equation*}
  The colored fan of $\hX_\n$ is illustrated in the following picture
  (for $\n = 2$).
  
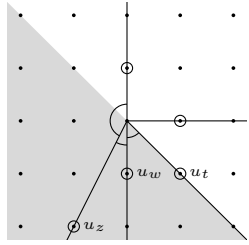
\begin{align*}
    \begin{tikzpicture}[scale=0.7]
      \clip (-2.24, -2.24) -- (2.24, -2.24) -- (2.24, 2.24) -- (-2.24, 2.24) -- cycle;
      \fill[color=gray!30] (-3, 3) -- (3, -3) -- (-3, -3) -- cycle;
      \foreach \x in {-3,...,3} \foreach \y in {-3,...,3} \fill (\x, \y) circle (1pt);
      \draw (1, 0) circle (3pt);
      \draw (0, 1) circle (3pt);
      \draw (1, -1) circle (3pt);
      \draw (-1, -2) circle (3pt);
      \draw (0, -1) circle (3pt);
      \draw (0,0) -- (3,0);
      \draw (0,0) -- (3,-3);
      \draw (0,0) -- (0,3);
      \draw (0,0) -- (-2, -4);
      \draw (0,0) -- (0, -4);
      \path (-1, -2) node[anchor=west] {{\tiny{$u_z$}}};
      \path (0, -1) node[anchor=west] {{\tiny{$u_w$}}};
      \path (1, -1) node[anchor=west] {{\tiny{$u_t$}}};
      \begin{scope}
        \clip (0,0) -- (0,1) -- (-1,1) -- (-1, -2) -- cycle; \draw (0,0) circle (9pt);
      \end{scope}
      \begin{scope}
        \clip (0,0) -- (-1, -2) -- (0, -2) -- cycle; \draw (0,0) circle (13pt);
      \end{scope}
      \begin{scope}
        \clip (0,0) -- (0, -1) -- (1, -2) -- (1, -1) -- cycle; \draw (0,0) circle (9pt);
      \end{scope}
    \end{tikzpicture}
  \end{align*}
  According to Proposition~\ref{prop:cox1}, we have
  \begin{align*}
    \Rm(\stX_\n) = \Qd[a,b,c,d,z,w,t]/\langle ad-bc-z^{\n+1}w \rangle
  \end{align*}
  with $\Pic(\stX_\n) \cong \Z^3$ where $\deg(z) = (1,0,0)$,
  $\deg(w) = (0,1,0)$, $\deg(t) = (0,0,1)$, $\deg(a) = \deg(c) = (1,0,-1)$, and
  $\deg(b) = \deg(d) = (\n, 1, 1)$. Moreover, the quasi-affine variety
  \begin{align*}
    \Tm_\n &\coloneqq \Spec(\Rm(\tX_\n)) \setminus (\Vd(a,c) \cup \Vd(b,d,w) \cup \Vd(z,t))
  \end{align*}
  admits a torsor $\Tm_\n \to \stX_\n$ for the action of the torus
  $\Spec(\Qd{}[\Pic(\stX_\n)]) \cong \Gd_m^3$.
\end{remark}

\begin{remark}\label{rem:singular_locus_counting}
  On the singular locus $X_{\n,\sing} = \Vd(a,c,z) \cong \Pd^1_\Qd$ with
  coordinates $(b:d)$, the height $H$ is the $\frac{\n+2}{2\n}$-th power of the
  standard anticanonical height on $\Pd^1_\Qd$. Therefore, we have
  \begin{equation*}
    N_{X_{\n,\sing}, H}(B) = \frac{2}{\zeta(2)}B^{\frac{2\n}{\n+2}} + O(B^{\frac{\n}{\n+2}}\log B)\text{.}
  \end{equation*}
  In particular,
  \begin{equation*}
    N_{X_2, H}(B) = N_{U, H}(B) + O(B)\text{.}
  \end{equation*}
\end{remark}

We are now going to state the counting problem for $X'_\n$. We consider the open
subset
\begin{align*}
  U' \coloneqq \stX'_\n \setminus \Vd(yzwt) = X'_\n \setminus \Vd(yzt)\text{.}
\end{align*}

\begin{prop}
  There is a natural $16$-to-$1$ correspondence between
  \begin{align*}
    \Um' \coloneqq \left\{(a,b,c,d,y,z,t,w) \in \Zd^8 :
    \begin{aligned}
      &yztw \ne 0 ;\  ad-bc-y^\n z^{\n+1}w=0\\
      &\gcd(a,c)=\gcd(z,t)=\gcd(y,t)=1\\
      &\gcd(b,d,z)=\gcd(b,d,w)=\gcd(y,w)=1
    \end{aligned}
        \right\}
  \end{align*}
  and the set $U'(\Qd)$. Moreover, for $(a,b,c,d,y,z,t,w) \in \Um'$, we have
  \begin{equation*}
    H'(\pi(a:b:c:d:y:z:t:w)) = \max|\Ms'_{\n}(a,b,c,d,y,z,t,w)|^{1/(\n+1)}\text{.}
  \end{equation*}
\end{prop}
\begin{proof}
  As Proposition~\ref{prop:41}.
\end{proof}

\begin{cor}\label{cor:countingproblem_b}
  We have that $N_{U',H'}(B^{1/(\n+1)})$ is equal to
  \begin{align*}
    \frac{1}{16} \#\left\{(a,b,c,d,y,z,t,w) \in \Zd^8 :
    \begin{aligned}
      &yztw \ne 0;\ ad-bc-y^\n z^{\n+1}w=0\\
      &\gcd(a,c)=\gcd(z,t)=\gcd(y,t)=1\\
      &\gcd(b,d,z)=\gcd(b,d,w)=\gcd(y,w)=1\\
      &\max|\Ms'_\n(a,b,c,d,y,z,t,w)| \le B
    \end{aligned}
        \right\}\text{.}
  \end{align*}
\end{cor}

\begin{remark}
  \label{rem:monomials}
  It is not difficult to see that we may assume that
  $\Ms'_{\n}(a,b,c,d,y,z,t,w)$ only contains the 13 monomials
  \begin{gather*}
    \pbox{\textwidth}
    { $\begin{aligned}
        &\{b, d\}^{\n+3}&\cdot&\{a, c\}^2&\cdot&y^2\text{,}\\
        &\{b,d\}^{\n+1}&\cdot&\{a,c\}^{2\n+2}&& && &\cdot& t^{2\n+2} &\cdot& w^2\text{,}\\
        &\{b, d\}^{\n+1}&&&\cdot&y^{2\n+2}&\cdot&z^{2\n+2} &&&\cdot& w^2\text{,}\\
        &&&\{a,c\}^{\n^2+2\n+1}&& &\cdot&z^{\n+1} &\cdot& t^{\n^2+3\n+2} &\cdot& w^{\n+3}\text{,}\\
        &&&&&y^{\n^2+2\n+1}&\cdot&z^{\n^2+3\n+2}&\cdot& t^{\n+1} &\cdot& w^{\n+3}\text{,}
      \end{aligned}$ }
\end{gather*}
where the notation $\{b, d\}$ resp.~$\{a, c\}$ means $b$ or $d$ resp.~$a$ or
$c$.
\end{remark}

\section{The expected formula for $X_2$}\label{sec:exp-x2}

The aim of this section is to determine the expected asymptotic formula for
$N_{U,H}(B)$ where
\begin{align*}
  U \coloneqq \tX_2 \setminus \Vd(w) = X_2 \setminus \Vd(a,c,z)\text{.}
\end{align*}
The resolution $\pi\colon \stX_2 \to X_2$ is crepant, hence the pullback of $H$
is an anticanonical height on $\stX_2$. According to
\cite[Conjecture~C']{MR1032922} and \cite[5.1]{MR2019019}, we have the predicted
asymptotic formula
\begin{align*}
  N_{U,H}(B) \sim \alpha\beta\tau B \log B\text{.}
\end{align*}
with
\begin{align*}
  \alpha = \rk \Pic(\tX_2) \cdot \vol\rleft\{t \in \Eff(\tX_2)^\vee: (t,-K_{\tX_2}) \le 1\rright\}
\end{align*}
where the volume is normalized such that $\Pic(\stX_2)^\vee$ has covolume $1$ in
$\Pic(\stX_2)^\vee_\Rd$. Under the identification $\Pic(\stX_2) \cong \Z^2$ from
Section~\ref{sec:geometry}, we have
\begin{align*}
  \alpha = 2\cdot\vol\rleft\{(t_1,t_2) \in \Rd_{\ge 0}^2: 4t_1+2t_2 \le 1\rright\} = \frac{1}{8}\text{.}
\end{align*}
The cohomological constant $\beta$ is
\begin{equation*}
  \beta = \#H^1(\Gal(\Qbar/\Qd),\Pic((\tX_2)_{\Qbar})) = 1
\end{equation*}
since $\tX_2$ is split. Finally, we determine the Tamagawa number $\tau$.
Consider the chart
\begin{align*}
  \Ad^3_\Qd \to \tX_2\text{,}&&
  (a,d,z) \mapsto (a: ad-z^3 : 1 : d :z : 1)\text{.}
\end{align*}
It follows from \cite[4.6]{MR2019019} and \cite[2.2.1]{MR1340296} that we have
\begin{align*}
  \tau = \omega_\infty \rleft(\prod_{p\text{ prime}} \lambda_p\omega_p\rright)
\end{align*}
with $\lambda_p = (1-p^{-1})^2$ and
\begin{align*}
  \omega_\nu &\coloneqq \iiint_{\Qd_\nu^3} \frac{1}{\max|\Ms_2(a,ad-z^3,1,d,z,1)|_\nu^2} \,\df a\,\df d\, \df z\text{,}
\end{align*}
for $\nu = p$ and $\nu = \infty$, where we have used the isomorphism
\begin{align*}
  \omega_{\tX_2} \cong \Om_{\tX_2}(-4, -2)
\end{align*}
identifying the section $\df a \wedge \df b \wedge \df z$ from the chart with the
section $1/c^4w^2$ from the Cox ring. Note that $\omega_p$ and $\omega_\infty$
(but not the product $\tau$) depend on the choice of such an isomorphism.

\begin{lemma}\label{lemma:pt}
  We have
  \begin{align*}
    \prod_{p\text{ prime}}\lambda_p\omega_p = \frac{1}{\zeta(2)\zeta(3)}\text{.}
  \end{align*}
\end{lemma}
\begin{proof}
  A direct calculation of the $p$-adic integral yields
  \begin{align*}
    \omega_p = \rleft(1+\frac{1}{p}\rright)\rleft(1+\frac{1}{p}+\frac{1}{p^2}\rright)\text{,} 
  \end{align*}
  from which the result follows. Alternatively, we may compute $\omega_p$ using
  an integral model of $\stX_2$. The toric variety $\stY_2$ comes from a regular
  fan, which can be used to construct a toric scheme $\stYf_2$ over
  $\Spec(\Zd)$, together with a map
  \begin{align*}
    \Spec(\Z[a,b,c,d,z,w]) \setminus (\Vd(a,c,z)\cup \Vd(b,d,w)) \to \stYf_2\text{,}
  \end{align*}
  which is a model for the universal torsor over
  \begin{align*}
    \stY_2 = \stYf_2 \times_{\Spec(\Zd)} \Spec(\Qd)\text{.}
  \end{align*}
  For details, we refer to \cite[Section~8]{MR1679841}. It can now be verified
  that the equation $ab-cd-z^{\n+1}w = 0$ defines a closed subscheme
  $\stXf_2 \hookrightarrow \stYf_2$, which is smooth and has integral fibers
  over $\Spec(\Zd)$ such that
  \begin{align*}
    \stX_2 = \stXf_2 \times_{\Spec(\Zd)} \Spec(\Qd)\text{.}
  \end{align*}
  We have the chart
  \begin{align*}
    \Ad^3_\Zd \to \tXf_2\text{,}&&
    (a,d,z) \mapsto \rleft(a: ad-z^3 : 1 : d :z : 1 \rright)
  \end{align*}
  and see that the isomorphism of line bundles
  \begin{align*}
    \omega_{\tXf_2} \cong \Om_{\tXf_2}(-4, -2)
  \end{align*}
  identifying $\df a \wedge \df b \wedge \df z$ and $1/c^4w^2$ can be defined
  over $\Spec(\Zd)$. Hence, according to \cite[Lemme~6.1]{pla}, we have
  \begin{equation*}
    \omega_p = \frac{\#\tXf_2(\Fd_p)}{p^2} = \frac{(p+1)(p^2+p+1)}{p^2} =
    \rleft(1+\frac{1}{p}\rright)\rleft(1+\frac{1}{p}+\frac{1}{p^2}\rright)\text{.}
    \qedhere
  \end{equation*}
\end{proof}

Finally, we compute the real density.

\begin{lemma}\label{le:real_density}
  We have
  \begin{equation*}
    \omega_\infty = 2 \iiiint_{|a|,|c|,|z|,|(ad-z^3)/c|,|d| \le 1} \frac{1}{|c|} \,\df a\,\df c\,\df d\,\df z\text{.}
  \end{equation*}
\end{lemma}

\begin{proof}
  We have
  \begin{align*}
    \omega_{\infty} = \iiint \frac{1}{\max\{|a^2|,|1|,|z^2|,|ad-z^3|,|d|\}^2}\,\df a\,\df d\, \df z\text{.}
  \end{align*}
  We introduce an additional integration over $c$ using the identity
  \begin{align*}
    \frac{1}{s} = \frac{1}{2} \int_{|c| \ge s} \frac{1}{|c^2|}\,\df c
  \end{align*}
  for $s \in \Rd_{>0}$ and obtain 
  \begin{align*}
    \omega_{\infty} = \frac{1}{2}\iiiint_{
    |a^2|,|1|,|z^2|,|ad-z^3|,|d| \le |c|^{1/2}}
    \frac{1}{|c^2|} \,\df a\,\df c\,\df d\,\df z\text{.}
  \end{align*}
  Now the transformation $c \mapsto \frac{1}{c^4}$ (with
  $\df c \mapsto \frac{4}{|c|^5} \, \df c$) yields
  \begin{align*}
    \omega_{\infty} = 2\iiiint_{
    |a^2c^2|,|c^2|,|z^2c^2|,|(ad-z^3)c^2|,|dc^2| \le 1}
    |c^3| \,\df a\,\df c\,\df d\,\df z\text{,}
  \end{align*}
  and, finally, the transformation
  $ (a,d,z) \mapsto (\frac{a}{c},\frac{d}{c^2},\frac{z}{c}) $ yields
  \begin{equation*}
    \omega_{\infty} = 2\iiiint_{
      |a^2|,|c^2|,|z^2|,|(ad-z^3)/c|,|d| \le 1}
    \frac{1}{|c|} \,\df a\,\df c\,\df d\,\df z\text{.}\qedhere
  \end{equation*}
\end{proof}

\section{The expected formula for $X_\n$ in the case $\n\ge 3$}
\label{sec:exp-ge-3}

The aim of this section is to determine, for $\n \ge 3$, the expected asymptotic
formula for $N_{U,H}(B)$, where
\begin{align*}
  U \coloneqq \tX_\n \setminus \Vd(w) = X_\n \setminus \Vd(a,c,z)
\end{align*}
and, moreover, to prove Theorem~\ref{thm:height_constant_intro}.

Recall from Section~\ref{sec:parameterization} that we consider
\begin{align*}
  L\coloneqq\tfrac{\n}{\n+2}\cdot \pi^*(-K_{X_\n}) = (\n,1) \in \Zd^2 \cong \Pic(\tX_\n)\text{,}
\end{align*}
and that the pullback of $H^{\n/(\n+2)}$ is a height relative to $L$. According
to \cite[Conjecture~C']{MR1032922} (see also \cite[3.6]{MR2019019}), the
predicted asymptotic formula is
\begin{align*}
  N_{U,H}(B) \sim \cfr B^\afr(\log B)^{\bfr-1}\text{,}
\end{align*}
where 
\begin{align*}
  \afr \coloneqq \tfrac{\n}{\n+2} \cdot \inf \rleft\{t\in\Rd : t \cdot L + K_{\tX_\n} \in \Pic(\tX_\n)
  \text{ is effective} \rright\} = \tfrac{2\n}{\n+2}
\end{align*}
and $\bfr=1$ is the codimension of the minimal face of the effective cone of
$\stX_\n$ containing $\smash{\afr \cdot L + K_{\tX_\n}}$.

Next, we compute the prediction of \cite{MR1679843} for $\cfr$. The divisor
\begin{align*}
  \tfrac{\n+2}{\n}\cdot \smash{\afr \cdot L + K_{\tX_\n}} = (\n-2,0)
\end{align*}
is not rigid, hence, according to \cite[Remark~2.4.4]{MR1679843}, we consider
the natural fibration
\begin{align*}
  \phi\colon \tX_\n \to P_\n \coloneqq \Proj\rleft(\bigoplus_{\nu\ge 0}\Gamma\rleft(\tX_\n, \Om_{\tX_\n}(\n-2, 0)^{\otimes \nu}\rright)\rright)\text{,}
\end{align*}
where we have an isomorphism $\Pd^2_\Qd \cong P_\n$ such that
\begin{align*}
  \phi \colon \tX_\n \to \Pd^2_\Qd\text{,}&&
  (a:b:c:d:z:w) \mapsto (a:c:z)\text{.}
\end{align*}
As we have $\fibphix \subseteq \Vd(w)$ if and only if $x \in \Vd(a, c)$, we only
consider points $x \in \Pd^2(\Qd) \setminus \Vd(a, c)$ and determine the
predicted asymptotic formula
\begin{align*}
  N_{\fibphix,\pi^*H}(B) \sim \cfr_xB^{\afr_x}(\log B)^{\bfr_x-1}
\end{align*}
for the fiber $\fibphix$. We have isomorphisms
\begin{align*}
  \Pd^1_\Qd &\to \fibphix\text{,}&
  (b:w) &\mapsto \rleft(a : b :c : \tfrac{bc+z^{\n+1}w}{a} :z: w \rright)\text{,} && \text{for $a \ne 0$,}\\
  \Pd^1_\Qd &\to \fibphix\text{,}&
  (d:w) &\mapsto \rleft(a : \tfrac{ad-z^{\n+1}w}{c} :c : d :z: w \rright)\text{,} && \text{for $c \ne 0$,}
\end{align*}
which depend on the choice of $a,c,z \in \Qd$ such that $x = (a:c:z)$. We now
see that $\smash{\pi^*H^{2\n/(\n+2)}}$ restricted to $\fibphix$ is an
anticanonical height on $\Pd^1_\Qd$, which means that the predicted asymptotic
formula is
\begin{align*}
  N_{\fibphix,\pi^*H}(B^{\frac{\n+2}{2\n}}) \sim \frac{1}{2}\omega_{\infty,x}\rleft(\prod_{p\text{ prime}}\lambda_p\omega_{p,x}\rright)B\text{,}
\end{align*}
where $\lambda_p = 1-p^{-1}$. Now, consider the charts
\begin{align*}
  \Ad^1_\Qd &\to \fibphix\text{,} & b &\mapsto \rleft(a : b :c : \tfrac{bc+z^{\n+1}}{a} :z: 1 \rright)\text{,} && \text{for $a \ne 0$,}\\
  \Ad^1_\Qd &\to \fibphix\text{,}&
  d &\mapsto \rleft(a : \tfrac{ad-z^{\n+1}}{c} :c : d :z: 1 \rright)\text{,} && \text{for $c \ne 0$.}
\end{align*}
According to \cite[2.2.1]{MR1340296}, we have
\begin{align*}
  \omega_{\nu,x} &\coloneqq
  \begin{dcases}
    \int_{\Qd_{\nu}} \frac{1}{|a| \max|\Ms_\n(a,b,c,(bc+z^{\n+1})/a,z,1)|_\nu^{2}}\,\df b & \text{for $a \ne 0$,}\\
    \int_{\Qd_{\nu}} \frac{1}{|c| \max|\Ms_\n(a,(ad-z^{\n+1})/c,c,d,z,1)|_\nu^{2}}\,\df d & \text{for $c\ne 0$}
  \end{dcases}
\end{align*}
for $\nu \coloneqq p$ and $\nu \coloneqq \infty$, where we have used the
isomorphism
\begin{align*}
  \omega_{\fibphix} \cong \Om_{\tX_\n}(-2L)|_{\fibphix}
\end{align*}
identifying the section $\df b$ from the first chart (resp.~the section $\df d$
from the second chart) with the section $a/w^2$ from the Cox ring (resp.~the
section $c/w^2$ from the Cox ring). Imposing the conditions $a,c,z \in \Zd$ and
$\gcd(a,c,z) =1$, the integrals $\omega_{\nu,x}$ only depend on
$x \in \Pd^2(\Qd) \setminus \Vd(a,c)$.

It follows that we have $\afr_x = \tfrac{2\n}{\n+2} = \afr$,
$\bfr_x = 1 = \bfr$, and
\begin{align*}
  \cfr_x = \frac{1}{2} \omega_{\infty,x}\prod_{p\text{ prime}}\lambda_p\omega_{p,x}\text{.}
\end{align*}
Summing over all the fibers, we obtain the expected constant
\begin{align*}
  \cfr = \sum_{x \in \Pd^2(\Qd)\setminus \Vd(a,c)} \cfr_x\text{.}
\end{align*}
in the asymptotic formula for $N_{U,H}(B)$. We show in
Corollary~\ref{cor:cx-finite} that this sum converges.

\begin{lemma}\label{lemma:pt_k}
  For every $x \in \Pd^2(\Qd)\setminus \Vd(a,c)$ we have
  \begin{align*}
    \prod_{p\text{ prime}}\lambda_p\omega_{p,x} = \frac{1}{\zeta(2)}\text{.}
  \end{align*}
\end{lemma}
\begin{proof}
  A straightforward calculation of the $p$-adic integrals yields
  \begin{align*}
    \omega_{p,x} = 1+\frac{1}{p}\text{,} 
  \end{align*}
  from which the result follows.
\end{proof}

\begin{lemma}\label{lemma:omega_infty_x}
  We have
  \begin{align*}
    \omega_{\infty,x} = 
    \begin{dcases}
      \iint_{\substack{\subalign{
            |a^{\n}w|,|c^{\n}w|,|z^{\n}w| &\le 1\\
            |b|,|(bc+z^{\n+1}w)/a| &\le 1
          }}}  \frac{1}{|a|} \,\df b\,\df w & \text{for $a \ne 0$,}\\
      \iint_{\substack{\subalign{
            |a^{\n}w|,|c^{\n}w|,|z^{\n}w| &\le 1\\
            |(ad-z^{\n+1}w)/c|,|d| &\le 1 }}} \frac{1}{|c|} \,\df d\,\df w
      &\text{for $c \ne 0$.}
    \end{dcases}
  \end{align*}
\end{lemma}
\begin{proof}
  We consider the case $a \ne 0$ (the case $c \ne 0$ is similar). In our
  expression for $\omega_{\nu,x}$ above, with $\nu = \infty$, we introduce an
  additional integration over $w$ as in Lemma~\ref{le:real_density} and obtain
\begin{align*}
  \omega_{\infty,x} = \frac{1}{2}\iint_{\substack{\subalign{
  |a^{\n}|,|c^{\n}|,|z^{\n}| &\le |w|^{1/2}\\
  |b|,|(bc+z^{\n+1})/a| &\le |w|^{1/2}
                          }}}  \frac{1}{|aw^2|} \,\df b\,\df w\text{.}
\end{align*}
Now the transformations $w \mapsto \frac{1}{w^2}$ and then
$ b \mapsto \frac{b}{w} $ give the result.
\end{proof}

The following result proves Theorem~\ref{thm:height_constant_intro}.

\begin{theorem}\label{thm:height_constant}
  Let $\overline{H}\colon \Pd^2(\Qd) \to \Rd_{>0}$ be a height relative to
  \begin{align*}
    \Om_{\Pd^2_{\Qd}}(-\n-1) \cong \Om_{P_\n}(-1) \otimes \omega_{P_\n}\text{.}
  \end{align*}
  There exist positive constants $c_1, c_2$ such that for every
  $x \in \Pd^2(\Qd) \setminus \Vd(a,c)$ we have
  \begin{align*}
    c_1 \overline{H}(x) \le \cfr_x \le c_2 \overline{H}(x)\text{.}
  \end{align*}
\end{theorem}
\begin{proof}
  Let $x\coloneqq (a:c:z) \in \Pd^2(\Qd)$ with integral coordinates and
  $\gcd(a,c,z) = 1$. Without loss of generality, we can define the height
  $\overline{H}$ as
  \begin{align*}
    \overline{H}(x) \coloneqq \frac{1}{\max\{|a|, |c|, |z|\}^{\n+1}}\text{.}
  \end{align*}
  First, assume $\max\{|a|,|c|,|z|\} = |a|$. Then we have
  \begin{align*}
    \omega_{\infty,x} &\le \iint_{|a^{\n}w|,|b| \le 1} \frac{1}{|a|} \,\df b\,\df w
                        = \int_{|a^{\n}w|\le 1} \frac{2}{|a|} \,\df w \ll \frac{1}{|a|^{\n+1}}\text{.}
  \end{align*}
  Moreover, the conditions $|a^{\n}w|,|b| \le \tfrac{1}{2}$ imply all the
  conditions on the integral $\omega_{\infty,x}$, so that we obtain
  \begin{align*}
    \omega_{\infty,x} &\ge \iint_{|a^{\n}w|,|b| \le \frac{1}{2}} \frac{1}{|a|} \,\df b\,\df w
                        = \frac{1}{|a|^{\n+1}}\text{.}
  \end{align*}
  The case $\max\{|a|,|c|,|z|\} = |c|$ is similar. It remains to consider
  $\max\{|a|,|c|,|z|\} = |z|$. Assume $|a| \ge |c|$. We have
  \begin{align*}
    \omega_{\infty,x} &\le \iint_{|b|,|(bc+z^{\n+1}w)/a|\le 1} \frac{1}{|a|} \,\df b\,\df w
                        = \int_{|b| \le 1} \frac{2|a|}{|az^{\n+1}|}\,\df b \ll \frac{1}{|z|^{\n+1}}\text{.}
  \end{align*}
  Moreover, the conditions $|b| \le \frac{|a|}{2|c|}$ and
  $|w| \le \frac{|a|}{2|z|^{\n+1}}$ imply all the conditions on the integral
  $\omega_{\infty,x}$, so that we obtain
  \begin{align*}
    \omega_{\infty,x} &\ge \iint_{|b| \le \frac{|a|}{2|c|},|w| \le \frac{|a|}{2|z|^{\n+1}}} \frac{1}{|a|} \,\df b\,\df w = \frac{|a|}{|cz^{\n+1}|} \ge \frac{1}{|z|^{\n+1}}\text{.}
  \end{align*}
  The case $|c| \ge |a|$ is similar. Together, we obtain $\overline{H}(x) \asymp \omega_{\infty,x} \asymp \cfr_x$.
\end{proof}

\begin{cor}
  \label{cor:cx-finite}
  We have
  \begin{align*}
    \sum_{\substack{x \in \Pd^2(\Qd)\setminus\Vd(a,c)}} \cfr_x < \infty\text{.}
  \end{align*}
\end{cor}
\begin{proof}
  Since $\n \ge 3$, we have
  \begin{equation*}
    \sum_{\substack{x \in \Pd^2(\Qd)\setminus\Vd(a,c)}} \cfr_x \ll \sum_{a,c,z} \frac{1}{\max\{|a|,|c|,|z|\}^{\n+1}} \ll 1\text{.}
    \qedhere
  \end{equation*}
\end{proof}

\section{The expected formula for $X'_\n$ in the case $\n \ge 2$}
\label{sec:exp-b-ge-2}

The aim of this section is to determine, for $\n \ge 2$, the expected asymptotic
formula for $N_{U',H'}(B)$ where
\begin{align*}
  U' \coloneqq \tX'_\n \setminus \Vd(yztw) = X'_\n \setminus \Vd(yzt)
\end{align*}
and, moreover, to prove Theorem~\ref{thm:height_constant_intro_b}.

Recall from Section~\ref{sec:parameterization} that we consider
\begin{align*}
  L'\coloneqq (\n+1) \cdot \pi^*(-K_{X'_\n}) = (\n^2+3\n+2, \n^2+2\n+1, \n+1, \n+3)
\end{align*}
in $\Pic(\tX'_\n) \cong \Zd^4$, and that the pullback of $(H')^{\n+1}$ is a
height relative to $L'$. According to \cite[Conjecture~C']{MR1032922} (see also
\cite[3.6]{MR2019019}), the predicted asymptotic formula is
\begin{align*}
  N_{U',H'}(B) \sim \cfr B^\afr(\log B)^{\bfr-1}\text{,}
\end{align*}
where 
\begin{align*}
  \afr \coloneqq (\n+1) \cdot \inf \rleft\{t\in\Rd : t \cdot L' + K_{\tX'_\n} \in \Pic(\tX'_\n)
  \text{ is effective} \rright\} = \tfrac{2\n+2}{\n+3}
\end{align*}
and $\bfr=1$ is the codimension of the minimal face of the effective cone of
$\stX'_\n$ containing $\smash{\afr \cdot L' + K_{\tX'_\n}}$. A prediction for
$\cfr$ can be found in \cite{MR1679843}. The $\Qd$-divisor
\begin{align*}
  \tfrac{1}{\n+1}\cdot \smash{\afr \cdot L'_\n + K_{\tX'_\n}} = \tfrac{\n-1}{\n+3} \cdot (\n+2,\n+1,1,0)
\end{align*}
is not rigid, hence, according to \cite[Remark~2.4.4]{MR1679843}, we consider
the natural fibration
\begin{align*}
  \phi'\colon \tX'_\n \dasharrow P'_\n \coloneqq \Proj\rleft(\bigoplus_{\substack{\nu\ge 0\\\n+3 | \nu}}\Gamma\rleft(\tX'_\n, \Om_{\tX'_\n}(\n+2,\n+1,1,0)^{\otimes \frac{\n-1}{\n+3}\cdot \nu}\rright)\rright)\text{,}
\end{align*}
where we have an isomorphism $\Pd^2_\Qd \cong P'_\n$ such that $\phi'$ extends
to
\begin{align*}
  \phi' \colon \tX'_\n \to \Pd^2_\Qd\text{,}&&
  (a:b:c:d:y:z:t:w) \mapsto (\ca:\cc:\cy) \coloneqq (at:ct:yz)\text{.}
\end{align*}
As we have $\fibphibx \subseteq \Vd(yztw)$ if and only if
$x \in \Vd(\ca, \cc) \cup \Vd(\cy)$, we only consider points
$x \in \Pd^2(\Qd) \setminus (\Vd(\ca, \cc) \cup \Vd(\cy))$ and determine the
predicted asymptotic formula
\begin{align*}
  N_{\fibphibx,\pi^*H'}(B) \sim \cfr_xB^{\afr_x}(\log B)^{\bfr_x-1}
\end{align*}
for the fiber $\fibphibx$. We have isomorphisms
\begin{align*}
  \Pd^1_\Qd &\to \fibphibx\text{,}&
  (b:w) &\mapsto \rleft(\ca : b : \cc : \tfrac{b\cc+\cy^{\n}w}{\ca} : \cy :1:1: w \rright)\text{,} && \text{for $\ca \ne 0$,}\\
  \Pd^1_\Qd &\to \fibphibx\text{,}&
  (d:w) &\mapsto \rleft(\ca : \tfrac{\ca d-\cy^{\n}w}{\cc} : \cc : d : \cy: 1:1: w \rright)\text{,} && \text{for $\cc \ne 0$,}
\end{align*}
which depend on the choice of $\ca,\cc,\cy \in \Qd$ such that
$x = (\ca:\cc:\cy)$. We now see that $\smash{(\pi^*H')^{(2\n+2)/(\n+3)}}$
restricted to $\fibphibx$ is an anticanonical height on $\Pd^1_\Qd$, which means
that the predicted asymptotic formula is
\begin{align*}
  N_{\fibphibx,\pi^*H'}(B^{\frac{\n+3}{2\n+2}}) \sim \frac{1}{2}\omega_{\infty,x}\rleft(\prod_{p\text{ prime}}\lambda_p\omega_{p,x}\rright)B\text{,}
\end{align*}
where $\lambda_p = 1-p^{-1}$. Now, consider the charts
\begin{align*}
  \Ad^1_\Qd &\to \fibphibx\text{,} & b &\mapsto \rleft(\ca : b : \cc : \tfrac{b\cc+\cy^{\n}}{\ca} : \cy : 1: 1:1 \rright)\text{,} && \text{for $\ca \ne 0$,}\\
  \Ad^1_\Qd &\to \fibphibx\text{,}&
  d &\mapsto \rleft(\ca : \tfrac{\ca d-\cy^{\n}}{\cc} : \cc : d : \cy : 1 : 1 : 1\rright)\text{,} && \text{for $\cc \ne 0$.}
\end{align*}
According to \cite[2.2.1]{MR1340296},
we have
\begin{align*}
  \omega_{\nu,x} &\coloneqq
  \begin{dcases}
    \int_{\Qd_{\nu}} \frac{1}{|\ca|_\nu \max|\Ms'_\n(\ca,b,\cc,(b\cc+\cy^{\n})/\ca,\cy,1,1,1)|_\nu^{2/(\n+3)}}\,\df b & 
    \text{for $\ca \ne 0$,}\\
    \int_{\Qd_{\nu}} \frac{1}{|\cc|_\nu \max|\Ms'_\n(\ca,(\ca d-\cy^{\n})/\cc,\cc,d,\cy,1,1,1)|_\nu^{2/(\n+3)}}\,\df d &
    \text{for $\cc\ne 0$}
  \end{dcases}
\end{align*}
for $\nu \coloneqq p$ and $\nu \coloneqq \infty$, where we have used the isomorphism
\begin{align*}
  \omega_{\fibphibx}^{\otimes(\n+3)} \cong \Om_{\tX'_\n}(-2L'_\n)|_{\fibphibx}
\end{align*}
identifying the section $(\df b)^{\otimes(\n+3)}$ from the first chart
(resp.~the section $(\df d)^{\otimes(\n+3)}$ from the second chart) with the
section $a^{\n+3}/w^{2\n+6}$ from the Cox ring (resp.~the section
$c^{\n+3}/w^{2\n+6}$ from the Cox ring). Imposing the conditions
$\ca,\cc,\cy \in \Zd$ and $\gcd(\ca,\cc,\cy) =1 $, the integrals
$\omega_{\nu,x}$ only depend on
$x \in \Pd^2(\Qd) \setminus (\Vd(\ca,\cc) \cup \Vd(\cy))$.

It follows that we have $\afr_x = \tfrac{2\n}{\n+2} = \afr$,
$\bfr_x = 1 = \bfr$, and
\begin{align*}
  \cfr_x = \frac{1}{2} \omega_{\infty,x}\prod_{p\text{ prime}}\lambda_p\omega_{p,x}\text{.}
\end{align*}
Summing over all the fibers, we obtain the expected constant
\begin{align*}
  \cfr = \sum_{x \in \Pd^2(\Qd)\setminus (\Vd(\ca,\cc)\cup\Vd(\cy))} \cfr_x\text{.}
\end{align*}
in the asymptotic formula for $N_{U',H'}(B)$. We show in
Proposition~\ref{prop:cxb-finite} that $\cfr < \infty$.

\begin{lemma}\label{lemma:pt_k_prime}
  Let $\efr \coloneqq \tfrac{-\n+1}{\n+3}$. For every
  $x \in \Pd^2(\Qd)\setminus (\Vd(\ca,\cc) \cup \Vd(\cy))$ we have
  \begin{align*}
    \omega_{p,x} = \bigg(\rleft(1-\frac{1}{p}\rright)\cdot \frac{1-(p^\efr)^{\nu_p(\cy)+1}}{1-p^\efr}+\frac{1}{p}+\frac{(p^\efr)^{\nu_p(\cy)}}{p}\bigg)
    \cdot (p^{\efr+1})^{\min(\nu_p(\ca), \nu_p(\cc))}\text{.}
  \end{align*}
\end{lemma}
\begin{proof}
  A lengthy, but straightforward calculation of the $p$-adic integrals yields
  \begin{align*}
    \omega_{p,x} =
    \begin{dcases}
      1 + \rleft(1-\frac{1}{p}\rright)\sum_{j=1}^{\nu_p(\cy)-1} (p^{\efr})^j + (p^\efr)^{\nu_p(\cy)} & \text{for $\nu_p(\cy) > 0$,} \\
      \rleft(1+\frac{1}{p}\rright) \cdot (p^{\efr+1})^{\min(\nu_p(\ca), \nu_p(\cc))} & \text{otherwise,}
    \end{dcases}
  \end{align*}
  from which the result follows.
\end{proof}

\begin{lemma}\label{lemma:omega_infty_x_prime}
  We have
  \begin{align*}
    \omega_{\infty,x} = 
    \begin{dcases}
      \iint_{\max|\Ms'_\n(\ca,b,\cc,(b\cc+\cy^\n w)/\ca,\cy,1,1,w)| \le 1} \frac{1}{|\ca|} \,\df b\,\df w & \text{for $\ca \ne 0$,}\\
      \iint_{\max|\Ms'_\n(\ca,(\ca d -\cy^\n w)/\cc, \cc, d,\cy,1,1,w)| \le 1} \frac{1}{|\cc|} \,\df d\,\df w & \text{for $\cc \ne 0$.}
    \end{dcases}
  \end{align*}
\end{lemma}
\begin{proof}
  This is completely analogous to Lemma~\ref{lemma:omega_infty_x}.
\end{proof}

The following result proves Theorem~\ref{thm:height_constant_intro_b}.

\begin{theorem}\label{thm:height_constant_bar}
  Let $\overline{H}\colon \Pd^2(\Qd) \to \Rd_{>0}$ be a height relative to an
  arbitrary line bundle. Then there are $\epsilon > 0$ such that there does
  not exist an open subset $V \subseteq \Pd^2(\Qd)$ with positive constants
  $c_1, c_2$ such that for every $x \in V$ we have
  \begin{align*}
    c_1 \overline{H}(x)^{1-\epsilon} \le \cfr_x \le c_2 \overline{H}(x)^{1+\epsilon}\text{.}
  \end{align*}
\end{theorem}
\begin{proof}
  Let $x\coloneqq (\ca:\cc:\cy) \in \Pd^2(\Qd)$ with integral coordinates and
  $\gcd(\ca,\cc,\cy) = 1$. Without loss of generality, we can define the height
  $\overline{H}$ as $\overline{H}(x) \coloneqq \max\{|\ca|, |\cc|, |\cy|\}^{r}$ for
  some $r \in \Zd$. We define
  \begin{align*}
    \omega_{\infty,x}^- &\coloneqq \int_{\Rd} \frac{1}{|\ca| \max|\Ms'_\n(\ca,2b,\cc,0,\cy,1,1,1)|^{2/(\n+3)}}\,\df b\text{,}\\
    \omega_{\infty,x}^+ &\coloneqq \int_{\Rd} \frac{1}{|\ca| \max|\Ms'_\n(\ca,b,\cc,0,\cy,1,1,1)|^{2/(\n+3)}}\,\df b\text{.}
  \end{align*}
  Clearly $\omega_{\infty,x} \le \omega_{\infty,x}^+ =
  2\omega_{\infty,x}^-$. For $|\ca| \ge
  \max\{2,|\cc|,|\cy^n|\}$, 
  we have $\omega_{\infty,x}^- \le \omega_{\infty, x}$.  Now we choose $\ca_0,
  \cc_0, \cy_0 \in \Zd$ with $\gcd(\ca_0,\cc_0,\cy_0) = 1$ and
  $\max\{|\ca_0|,|\cc_0|,|\cy_0|\}=|\ca_0|$, and a prime $p$ with $p \nmid
  \ca_0,\cc_0,\cy_0$ such that
  \begin{align*}
    x_0(m) &\coloneqq (\ca_0p^m:\cc_0:\cy_0)\text{,}\\
    x'_0(m) &\coloneqq (\ca_0p^m:\cc_0p^m:\cy_0)
  \end{align*}
  lie in a given $V$ for all $m \gg 1$. Here and in the following, all
  implicit constants may depend on $\ca_0$, $\cc_0$, $\cy_0$, and $r$, but not
  on $m$. We have
  \begin{equation*}
    \overline{H}(x'_0(m)) = \overline{H}(x_0(m)) = (p^m|\ca_0|)^r\text{.}
  \end{equation*}
  For $m \gg 1$, we
  have
  \begin{align*}
    \cfr_{x_0(m)} &\asymp \omega_{\infty,x_0(m)} \asymp \omega_{\infty,x_0(m)}^- \asymp \omega_{\infty,x'_0(m)}^-\text{,}
  \end{align*}
  since $\omega_{\infty,x}^-$ does not depend on $\cc$ for $|\ca| \ge |\cc|$. We
  also have
  \begin{align*}
    \cfr_{x'_0(m)} &\asymp (p^m)^{\efr+1} \omega_{\infty,x'_0(m)} \asymp (p^m)^{\efr+1} \omega_{\infty,x'_0(m)}^- \asymp (p^m)^{\efr+1}\cfr_{x_0(m)}\text{.}
  \end{align*}
  
  Assume that, for all $m \ge 0$, we have
  \begin{equation*}
    \cfr_{x_0(m)} \gg \overline{H}(x_0(m))^{1-\epsilon},\qquad \cfr_{x'_0(m)}
    \ll \overline{H}(x'_0(m))^{1+\epsilon}.
  \end{equation*}
  Then 
  \begin{equation*}
    (p^m)^{\efr+1}(p^m|\ca_0|)^{r(1-\epsilon)}
    \ll (p^m)^{\efr+1}\cfr_{x_0(m)} \ll \cfr_{x'_0(m)} \ll
    (p^m|\ca_0|)^{r(1+\epsilon)}\text{,}
  \end{equation*}
  which implies
  \begin{equation*}
    (p^m)^{\efr+1-2r\epsilon} \ll |\ca_0|^{2r\epsilon}.
  \end{equation*}
  For $\efr+1-2r\epsilon > 0$ (i.e., $\epsilon < \frac{2}{r(n+3)}$) and $m \to
  \infty$, we arrive at a contradiction.
\end{proof}

\section{Estimating integral points on the universal torsor of
  $X_\n$}\label{sec:summations}

We are going to prove Theorem~\ref{thm:main-2_intro} by showing
\begin{equation*}
  N_{U,H}(B^2) = \cfr B^2 \log(B^2) + O(B^2)\text{,}
\end{equation*}
where $\cfr$ is as in Section~\ref{sec:exp-x2}, and we are going to prove
Theorem~\ref{thm:ge3_intro} by showing
\begin{equation*}
  N_{U,H}(B^{\frac{\n+2}{\n}}) = \cfr B^2 + O(B^{\frac{\n+2}{\n}})
\end{equation*}
for $\n \ge 3$, where $\cfr$ is as in Section~\ref{sec:exp-ge-3}.

We use \cite[Lemma~3.1]{der09} and \cite[Lemma~3.6]{df14} repeatedly to
approximate sums by integrals. Note that we have
\begin{align*}
  \max|\Ms_\n(a,b,c,d,z,w,t)| = \max\{|b|, |d|, |a^\n wt^{\n+1}|, |c^\n wt^{\n+1}|, |z^\n wt|\}\text{.}
\end{align*}
We define
\begin{align*}
  V_1(a,c,z,w; B) \coloneqq
  \begin{dcases}
    \int_{\substack{\subalign{|d|,|(ad-z^{\n+1}w)/c|&\le B}}} \frac{1}{|c|}\, \df d & \text{for $c\ne 0$,} \\
    \int_{\substack{\subalign{|b|,|(bc+z^{\n+1}w)/a|&\le B}}} \frac{1}{|a|}\, \df b & \text{for $a\ne 0$.} \
  \end{dcases}
\end{align*}
Note that for $ac\ne 0$ the two cases coincide. Moreover, we have
\begin{align*}
  V_1(a,c,z,w; B) \le \frac{B}{\max\{|a|,|c|\}}\text{.}
\end{align*}
We also define
\begin{align*}
  V_2(a,c,z; B) \coloneqq \int_{\substack{\subalign{
  1 \le |w|&\le B\\
  |a^{\n}w|&\le B\\
  |c^{\n}w|&\le B\\
  |z^{\n}w| &\le B\\
  }}} V_1(a,c,z,w; B) \,\df w\text{.}
\end{align*}

\begin{prop}
  \label{prop:mi1}
  For $\n \ge 2$, we have
  \begin{align*}
    N_{U,H}(B^{\frac{\n+2}{\n}}) 
    &=
      \frac{1}{4\zeta(2)} 
      \sum_{\substack{|a|,|c|,|z| \ge 0\\
    \mathclap{\gcd(a,c,z)=1}\\
    (a,c)\ne(0,0)
    }}
    V_2(a,c,z; B)
    +O(B^{\frac{\n+2}{\n}})
    \text{.}
  \end{align*}
\end{prop}
\begin{proof}
  By Corollary~\ref{cor:countingproblem}, we have
  \begin{align*}
    N_{U,H}(B^{\frac{\n+2}{\n}}) = \frac{1}{8}
    \sum_{\substack{|a|,|c|,|z| \ge 0\\
    |w|,|t|\ge 1\\
    \gcd(a,c)=\gcd(z,t)=1\\
    \mathclap{|a^{\n}wt^{\n+1}|,|c^{\n}wt^{\n+1}|,|z^{\n}wt| \le B}
    }}
    \#\rleft\{(b, d) \in \Zd^2 : \pbox{\textwidth}{$\begin{aligned}ad-bc &= z^{\n+1}w\\|b|, |d| &\le B\\\gcd(b,d,w)&=1\end{aligned}$}\rright\}\text{.}
  \end{align*}
  We apply a M\"obius inversion to the condition $\gcd(b,d,w)=1$ and obtain
  \begin{align*}
    N_{U,H}(B^{\frac{\n+2}{\n}}) &= \sum_{\ma \ge 1}\frac{\mu(\ma)}{8}
                                   \sum_{\substack{|a|,|c|,|z| \ge 0\\
    |w|,|t|\ge 1\\
    \gcd(a,c)=\gcd(z,t)=1\\
    \mathclap{|\ma a^{\n}wt^{\n+1}|,|\ma c^{\n}wt^{\n+1}|,|\ma z^{\n}wt| \le B}
    }}
    \#\rleft\{(b, d) \in \Zd^2 : 
    \pbox{\textwidth}{$\begin{aligned}ad-bc &= z^{\n+1}w\\ |b|, |d| &\le B/\ma\end{aligned}$}\rright\}\\
                                 &= \sum_{\ma \ge 1}\frac{\mu(\ma)}{8}
                                   \sum_{\substack{|a|,|c|,|z| \ge 0\\
    |w|,|t|\ge 1\\
    \gcd(a,c)=\gcd(z,t)=1\\
    \mathclap{|\ma a^{\n}wt^{\n+1}|,|\ma c^{\n}wt^{\n+1}|,|\ma z^{\n}wt| \le B}
    }}
    \rleft( V_1(a,c,z,w; B/\ma) + O(1)\rright)\\
                                 &= \sum_{\ma \ge 1}\frac{\mu(\ma)}{8}
                                   \sum_{\substack{|a|,|c|,|z| \ge 0\\
    |w|,|t|\ge 1\\
    \gcd(a,c)=\gcd(z,t)=1\\
    \mathclap{|\ma a^{\n}wt^{\n+1}|,|\ma c^{\n}wt^{\n+1}|,|\ma z^{\n}wt| \le B}
    }}
    \rleft(\frac{1}{\ma} V_1(a,c,z,\ma w; B) + O(1)\rright)\\
                                 &= \sum_{\ma \ge 1}\frac{\mu(\ma)}{8\ma}
                                   \sum_{\substack{|a|,|c|,|z| \ge 0\\
    |w|,|t|\ge 1\\
    \mathclap{\gcd(a,c)=\gcd(z,t)=1}\\
    \mathclap{|\ma a^{\n}wt^{\n+1}|,|\ma c^{\n}wt^{\n+1}|,|\ma z^{\n}wt| \le B}
    }}
    V_1(a,c,z,\ma w; B) + O(B^{\frac{\n+2}{\n}})
    \text{,}
  \end{align*}
  where the condition $\gcd(a,c)=1$ is used for the second equality, the
  transformation $\ma b \mapsto b$ or $\ma d \mapsto d$ is applied inside the
  integral $V_1$ for the third equality, and the fourth equality follows from
  the estimate
  \begin{align*}
    \sum_{\ma}\sum_{a,c,z,w,t} 1 &\ll 
                                   \sum_{\ma}\sum_{a,c,z,t} \frac{B}{\ma \max\{|a|,|c|\}^{\n}|t|^{\n+1}}
    \\&\ll 
        \sum_{\ma}\sum_{a,c,t} \frac{B^{(\n+1)/\n}}{\ma^{(\n+1)/\n}\max\{|a|,|c|\}^{\n}|t|^{\n+2/\n}}
    \\&\ll 
        \begin{cases}
          B^{(\n+1)/\n}\log B & \text{for $\n=2$}\\ B^{(\n+1)/\n} & \text{for $\n \ge 3$}
        \end{cases}
    \\&\ll B^{\frac{\n+2}{\n}}\text{.}
  \end{align*}
  Replacing the sum over $w$ by an integral, we obtain that
  $N_{U,H}(B^{\frac{\n+2}{\n}})$ is equal to
  \begin{align*}
    &\sum_{\ma \ge 1}\frac{\mu(\ma)}{8\ma} 
      \sum_{\substack{|a|,|c|,|z| \ge 0\\
    |t|\ge 1\\
    \mathclap{\gcd(a,c)=\gcd(z,t)=1}\\
    \subalign{|\ma a^{\n}t^{\n+1}|&\le B\\|\ma c^{\n}t^{\n+1}|&\le B\\|\ma z^{\n}t|&\le B}
                                                                                     }}
                                                                                     \Bigg(\int_{\substack{\subalign{
                                                                                     |w| &\ge 1\\
    |\ma a^{\n}wt^{\n+1}|&\le B\\
    |\ma c^{\n}wt^{\n+1}|&\le B\\
    |\ma z^{\n}wt| &\le B\\
    }}} V_1(a,c,z,\ma w; B)\,\df w   + O(\err_1)\Bigg) + O(B^{\frac{\n+2}{\n}})\\
    =& \sum_{\ma \ge 1}\frac{\mu(\ma)}{8\ma} 
       \sum_{\substack{|a|,|c|,|z| \ge 0\\
    |t|\ge 1\\
    \mathclap{\gcd(a,c)=\gcd(z,t)=1}\\
    \subalign{|\ma a^{\n}t^{\n+1}|&\le B\\|\ma c^{\n}t^{\n+1}|&\le B\\|\ma z^{\n}t|&\le B}
                                                                                     }}
                                                                                     \Bigg(\frac{1}{\ma}\int_{\substack{\subalign{
                                                                                     |w| &\ge \ma\\
    |a^{\n}wt^{\n+1}|&\le B\\
    |c^{\n}wt^{\n+1}|&\le B\\
    |z^{\n}wt| &\le B\\
    }}} V_1(a,c,z,w; B)\,\df w + O(\err_1)\Bigg) + O(B^{\frac{\n+2}{\n}})\\
    =& \sum_{\ma \ge 1}\frac{\mu(\ma)}{8\ma^2} 
       \sum_{\substack{|a|,|c|,|z| \ge 0\\
    |t|\ge 1\\
    \gcd(a,c)=\gcd(z,t)=1
    }}
    \int_{\substack{\subalign{
    |w|&\ge \ma\\
    |a^{\n}wt^{\n+1}|&\le B\\
    |c^{\n}wt^{\n+1}|&\le B\\
    |z^{\n}wt| &\le B\\
    }}} V_1(a,c,z,w; B) \,\df w + O(B^{\frac{\n+2}{\n}})\text{,}
  \end{align*}
  where we have applied the transformation $\ma w \mapsto w$
  for the first equality and the second equality follows from the estimates 
  \begin{align*}
    \err_1 = \max_w
    V_1(a,c,z,w; B)
    \le
    \frac{B}{\max\{|a|,|c|\}}
  \end{align*}
  and
  \begin{align*}
    \sum_\ma \frac{1}{\ma}\sum_{a,c,z,t} \frac{B}{\max\{|a|,|c|\}} &\ll
    \sum_\ma \frac{1}{\ma}\sum_{a,c,t} \frac{B^{(\n+1)/\n}}{\ma^{1/\n}\max\{|a|,|c|\}|t|^{1/\n}} 
    \\
    &\ll \sum_\ma \frac{1}{\ma}\sum_{t} \frac{B^{(\n+2)/\n}}{\ma^{2/\n}|t|^{(\n+2)/\n}}
    \ll B^\frac{\n+2}{\n}\text{.}
  \end{align*}
  Next, we replace the condition $\ma \le |w|$ by the condition
  $|t^{-1}| \le |w|$. For $0 < \epsilon < \frac{1}{\n}$, we have
  \begin{align*}
    &\sum_\ma \frac{1}{\ma^2}\sum_{a,c,z,t}
      \int_{\substack{\subalign{
      |t^{-1}| \le |w| &\le \ma\\
    |a^{\n}wt^{\n+1}|&\le B\\
    |c^{\n}wt^{\n+1}|&\le B\\
    |z^{\n}wt| &\le B
                 }}} V_1(a,c,z,w; B)\,\df w
    \\
    &\quad \ll  
      \sum_\ma \frac{1}{\ma^2}\sum_{a,c,z,t}
      \int_{\substack{\subalign{
      |t^{-1}| \le |w| &\le \ma\\
    |a^{\n}wt^{\n+1}|&\le B\\
    |c^{\n}wt^{\n+1}|&\le B\\
    |z^{\n}wt| &\le B
                 }}} \frac{B}{\max\{|a|,|c|\}} \,\df w
    \\
    &\quad \ll  
      \sum_\ma \frac{1}{\ma^2}\sum_{\substack{a,c,z,t\\z \ne 0\\\mathclap{|at|^{\n},
    |ct|^{\n},|z|^{\n} \le B}}}
    \frac{\ma^{1-\epsilon}B^{1+\epsilon}}{\max\{|a|,|c|\}|z|^{\epsilon \n}|t|^{\epsilon}}
    +
    \sum_\ma \frac{1}{\ma^2}\sum_{\substack{a,c,t\\z=0\\\mathclap{|at|^{\n},
    |ct|^\n \le B}}}
    \frac{\ma^{1-\epsilon}B^{1+\epsilon}}{\max\{|a|,|c|\}^{1+\epsilon \n}|t|^{\epsilon (\n+1)}}
    \\
    &\quad \ll  
      \sum_\ma \frac{1}{\ma^2}\sum_{\substack{a,c,t\\\mathclap{|at|^{\n},
    |ct|^{\n}\le B}}}
    \frac{\ma^{1-\epsilon}B^{(\n+1)/\n}}{\max\{|a|,|c|\}|t|^{\epsilon}}
    \ll  
    \sum_\ma \frac{1}{\ma^{1+\epsilon}}\sum_{t}
    \frac{B^{(\n+2)/\n}}{|t|^{1+\epsilon}}
    \ll B^{\frac{\n+2}{\n}}\text{,}
  \end{align*}
  so that we obtain that $N_{U,H}(B^{\frac{\n+2}{\n}})$ is equal to
  \begin{align*}
    &\frac{1}{8\zeta(2)} 
      \sum_{\substack{|a|,|c|,|z| \ge 0\\
    |t|\ge 1\\
    \gcd(a,c)=\gcd(z,t)=1
    }}
    \int_{\substack{\subalign{
    |wt|&\ge 1\\
    |a^{\n}wt^{\n+1}|&\le B\\
    |c^{\n}wt^{\n+1}|&\le B\\
    |z^{\n}wt| &\le B\\
    }}} V_1(a,c,z,w; B) \,\df w 
    +O(B^{\frac{\n+2}{\n}}) \\
    =&
       \frac{1}{8\zeta(2)} 
       \sum_{\substack{|a|,|c|,|z| \ge 0\\
    |t|\ge 1\\
    \gcd(a,c)=\gcd(z,t)=1
    }}
    \int_{\substack{\subalign{
    |w|&\ge 1\\
    |(at)^{\n}w|&\le B\\
    |(ct)^{\n}w|&\le B\\
    |z^{\n}w| &\le B\\
    }}} \frac{1}{|t|}V_1(a,c,z,w/t; B) \,\df w 
    +O(B^{\frac{\n+2}{\n}}) \\
    =&
       \frac{1}{8\zeta(2)} 
       \sum_{\substack{|a|,|c|,|z| \ge 0\\
    |t|\ge 1\\
    \gcd(a,c)=\gcd(z,t)=1
    }}
    \int_{\substack{\subalign{
    |w|&\ge 1\\
    |(at)^{\n}w|&\le B\\
    |(ct)^{\n}w|&\le B\\
    |z^{\n}w| &\le B\\
    }}} V_1(ta,tc,z,w; B) \,\df w 
    +O(B^{\frac{\n+2}{\n}}) \\
    =&
       \frac{1}{4\zeta(2)} 
       \sum_{\substack{|a|,|c|,|z| \ge 0\\
    \gcd(a,c,z)=1\\
    (a,c)\ne(0,0)
    }}
    \int_{\substack{\subalign{
    |w|&\ge 1\\
    |a^{\n}w|&\le B\\
    |c^{\n}w|&\le B\\
    |z^{\n}w| &\le B\\
    }}} V_1(a,c,z,w; B) \,\df w 
    +O(B^{\frac{\n+2}{\n}})\text{,}
  \end{align*}
  where the transformation $tw \mapsto w$ is applied for the first equality and
  the $2$-to-$1$ substitution $(ta, tc) \mapsto (a, c)$ is applied for the third
  equality. Finally, adding the condition $|w| \le B$ leaves the integral
  unchanged.
\end{proof}

We define
\begin{align*}
  V'_2(a,c,z) \coloneqq \int_{\substack{\subalign{
  |a^{\n}w|&\le 1\\
  |c^{\n}w|&\le 1\\
  |z^{\n}w| &\le 1\\
  }}} V_1(a,c,z,w; 1) \,\df w\text{.}
\end{align*}

\begin{cor}\label{cor:sum_omega_x}
  For $\n \ge 2$, we have
  \begin{align*}
    N_{U,H}(B^{\frac{\n+2}{\n}}) &= \frac{1}{4\zeta(2)}
    \sum_{\substack{|a|,|c|,|z| \le B^{1/\n}\\\gcd(a,c,z)=1\\(a,c) \ne
    (0,0)}}
    V'_2(a,c,z) B^2 + O(B^{\frac{\n+2}{\n}})\text{.}
  \end{align*}
\end{cor}

\begin{proof}
  In the formula from Proposition~\ref{prop:mi1}, we may restrict the sum to
  $|a|,|c|,|z| \le B^{1/\n}$ since $V_2(a,c,z;B)$ vanishes otherwise. The
  transformations $(b,w) \mapsto (Bb,Bw)$ for $a \ne 0$ and
  $(d,w)\mapsto (Bd,Bw)$ for $c \ne 0$ show that we have
  \begin{equation*}
    \int_{\substack{\subalign{
          |a^{\n}w|&\le B\\
          |c^{\n}w|&\le B\\
          |z^{\n}w| &\le B\\
        }}} V_1(a,c,z,w; B) \,\df w = V'_2(a,c,z)B^2\text{.}
  \end{equation*}
  Comparing the left side with $V_2(a,c,z;B)$, we see that the condition
  $|w| \le B$ in the definition of $V_2(a,c,z;B)$ follows from the other
  conditions, and it remains to remove the condition $|w| \ge 1$. The corollary
  now follows from the computation
  \begin{equation*}
    \sum_{|a|,|c|,|z| \le B^{1/\n}} \int_{|w| \le 1} V_1(a,c,z,w;B)\,\df w \ll
    \sum_{|a|,|c|,|z| \le B^{1/\n}} \frac{B}{\max\{|a|,|c|\}} \ll B^{\frac{\n+2}{\n}}\text{.}\qedhere
  \end{equation*}
\end{proof}

\begin{remark}\label{rem:order_of_magnitude}
  For $\n \ge 3$, we have $V'_2(a,c,z) = \omega_{\infty, (a:c:z)}$ 
  according to Lemma~\ref{lemma:omega_infty_x} and hence
  \begin{align*}
    V'_2(a,c,z) \asymp \frac{1}{\max\{|a|,|c|,|z|\}^{\n+1}}\text{,}
  \end{align*}
  but this also holds for $\n=2$.
  Theorem~\ref{thm:height_constant} and Corollary~\ref{cor:sum_omega_x} now yield
  \begin{equation*}
    N_{U,H}(B^{\frac{\n+2}{\n}}) \asymp \sum_{\substack{|a|,|c|,|z| \le B^{1/\n}\\\gcd(a,c,z)=1\\(a,c) \ne
        (0,0)}} \frac{B^2}{\max\{|a|,|c|,|z|\}^{\n+1}} \asymp
    \begin{cases}
      B^2 \log B\text{,} &\n=2\text{,}\\
      B^2\text{,} &\n \ge 3\text{.}
    \end{cases}
  \end{equation*}
  In the following, we turn these upper and lower bounds into asymptotic
  formulas.
\end{remark}

We begin with the case $\n \ge 3$.

\begin{theorem}\label{thm:ge3-final}
  For $\n \ge 3$, we have
  \begin{align*}
    N_{U,H}(B^{\frac{\n+2}{\n}}) = \cfr B^2 + O(B^{\frac{\n+2}{\n}})\text{.}
  \end{align*}
\end{theorem}
\begin{proof}
  We remove the conditions $|a|,|c|,|z| \le B^{1/\n}$ from the sum in
  Corollary~\ref{cor:sum_omega_x} with a satisfactory error term. Since
  $\omega_{\infty,(a:c:z)} \ll \max\{|a|,|c|,|z|\}^{-\n-1}$ by
  Theorem~\ref{thm:height_constant}, we have
  \begin{equation*}
    \sum_{\substack{\max\{|a|,|c|,|z|\} > B^{1/\n}}}
    \omega_{\infty,(a:c:z)}B^2 \ll \sum_{\substack{|a|\le|c|\le|z|\\|z| >
        B^{1/\n}}} \frac{B^2}{|z|^{\n+1}} \ll \sum_{|z| >
      B^{1/\n}} \frac{B^2}{|z|^{\n-1}} \ll B^{\frac{\n+2}{\n}}\text{.}
  \end{equation*}
  It follows that we have
  \begin{align*}
    N_{U,H}(B^{\frac{\n+2}{\n}}) 
    &=
      \frac{1}{4\zeta(2)}\Bigg(\sum_{\substack{|a|,|c|,|z| \ge 0\\
    \mathclap{\gcd(a,c,z)=1}\\(a,c) \ne (0,0)}}
    \omega_{\infty,(a:c:z)}\Bigg) B^2 + O(B^{\frac{\n+2}{\n}})\\
    &=\Bigg(\sum_{\substack{x \in \Pd^2(\Qd) \setminus \Vd(a,c)
      }}
      \frac{\omega_{\infty,x}}{2\zeta(2)}\Bigg)B^2 +O(B^{\frac{\n+2}{\n}})\text{,}
  \end{align*}
  as predicted in Section~\ref{sec:exp-ge-3}.
\end{proof}

We now turn to the case $\n=2$. Note that we already have a result on the order
of magnitude in Remark~\ref{rem:order_of_magnitude} following from
Corollary~\ref{cor:sum_omega_x}. In order to obtain an asymptotic formula, we
resume our calculation from Proposition~\ref{prop:mi1}. We define
\begin{align*}
  V_3(B) \coloneqq \iiint V_2(a,c,z; B)\,\df a\,\df c\,\df z\text{.} 
\end{align*}

\begin{lemma}\label{lem:volume_real_density}
  For $\n=2$, we have
  \begin{align*}
    V_3(B) = \omega_\infty B^2\log B  \text{.}
  \end{align*}
\end{lemma}

\begin{proof}
  We have
  \begin{align*}
    V_3(B) = \idotsint_{\substack{
    \subalign{1\le|w|&\le B\\
    |a^2w|,|c^2w|,|z^2w| &\le B\\
    |(ad-z^3w)/c|,|d|&\le B}}}
                       \frac{1}{|c|}\, \df a\,\df c\,\df d\,\df z\,\df w\text{.} 
  \end{align*}
  Applying the transformations $(a,c,z) \mapsto B^{1/2}/|w|^{1/2}(a,c,z)$
  and $b \mapsto Bb$, we obtain
  \begin{align*}
    V_3(B) &= B^2 \idotsint_{\substack{
             \subalign{1\le|w|&\le B\\
    |a^2|,|c^2|,|z^2| &\le B\\
    |(ad-z^3)/c|,|d|&\le B}}}
                      \frac{1}{|cw|}\, \df a\,\df c\,\df d\,\df z\,\df w\\
           &= B^2 \int_{1\le|w|\le B} \frac{1}{|w|}\,\df w  \iiiint_{\substack{\subalign{|a^2|,|c^2|,|z^2|
                              &\le 1\\|(ad-z^3)/c|,|d|&\le 1}}} \frac{1}{|c|}\, \df a\,\df c\,\df d\,\df z
                                                        \text{.} 
  \end{align*}
  Now the transformation $|w| \mapsto B^{|w|}$ (with $\df w \mapsto B^{|w|} \log B\,\df w$) yields
  \begin{align*}
    V_3(B) 
    &= B^2 \log B \int_{0\le|w|\le 1} \,\df w \iiiint_{\substack{\subalign{|a^2|,|c^2|,|z^2|
    &\le 1\\|(ad-z^3)/c|,|d|&\le 1}}} \frac{1}{|c|}\, \df a\,\df c\, \df d\,\df z\\
    &= 2 B^2 \log B \iiiint_{\substack{\subalign{|a^2|,|c^2|,|z^2|
    &\le 1\\|(ad-z^3)/c|,|d|&\le 1}}} \frac{1}{|c|}\, \df a\,\df c\,\df d\,\df z\\
    &= \omega_\infty B^2 \log B
      \text{,} 
  \end{align*}
  where the last equality follows from Lemma~\ref{le:real_density}.
\end{proof}

\begin{theorem}\label{thm:final_summations}
  For $\n=2$, we have
  \begin{align*}
    N_{U,H}(B^2) = \cfr B^2\log(B^2) + O(B^2)\text{.}
  \end{align*}
\end{theorem}

\begin{proof}
  According to Proposition~\ref{prop:mi1}, we have
  \begin{align*}
    N_{U,H}(B^2) 
    &=
      \frac{1}{4\zeta(2)} 
      \sum_{\substack{|a|,|c|,|z| \ge 0\\
    \mathclap{\gcd(a,c,z)=1}\\
    (a,c)\ne(0,0)
    }}
    V_2(a,c,z; B)
    +O(B^2)
    \text{.}
  \end{align*}
  We apply a M\"obius inversion to the condition $\gcd(a,c,z) = 1$ and replace
  the sum over $z$ by an integral to obtain
  \begin{align*}
    N_{U,H}(B^2) 
    &=
      \sum_{\ma\ge 1}\frac{\mu(\ma)}{4\zeta(2)} 
      \sum_{\substack{|a|,|c|,|z| \ge 0\\
    \mathclap{(a,c)\ne(0,0)}
    }}
    V_2(\ma a,\ma c,\ma z; B)
    +O(B^2)
    \\
    &=
      \sum_{\ma\ge 1}\frac{\mu(\ma)}{4\zeta(2)} 
      \sum_{\substack{|a|,|c| \ge 0\\
    \mathclap{(a,c)\ne(0,0)}}}
    \rleft(\int V_2(\ma a,\ma c,\ma z; B)\,\df z + O(\err_2) \rright) + O(B^2)
    \\
    &=
      \sum_{\ma\ge 1}\frac{\mu(\ma)}{4\zeta(2)\ma} 
      \sum_{\substack{|a|,|c| \ge 0\\
    \mathclap{(a,c)\ne(0,0)}}}
    \int V_2(\ma a,\ma c, z; B)\,\df z + O(B^2)
    \\
    &=
      \sum_{\ma\ge 1}\frac{\mu(\ma)}{4\zeta(2)\ma} 
      \sum_{|a|,|c| \ge 1}
      \int V_2(\ma a,\ma c, z; B)\,\df z + O(B^2)\text{,}
  \end{align*}
  where the third equality follows from the transformation $\ma z \mapsto z$ and
  the estimates
  \begin{align*}
    \err_2 \le \max_z \int_{\substack{\subalign{
    1 \le |w|&\le B\\
    |\ma^2 a^2w|&\le B\\
    |\ma^2 c^2w|&\le B\\
    |\ma^2 z^2w| &\le B\\
    }}} \frac{B}{\max\{|\ma a|,|\ma c|\}} \,\df w
    \ll
    \frac{B^2}{\ma^3\max\{|a|,|c|\}^3}
  \end{align*}
  and
  \begin{align*}
    \sum_{\ma,a,c} \frac{B^2}{\ma^3\max\{|a|,|c|\}^3} \ll B^2\text{.}
  \end{align*}
  Note that for every $a,c \in \Rd$ with $(a,c) \ne (0,0)$ we have
  \begin{align*}
    \int V_2(a,c,z; B) \,\df z &\le \iint_{\substack{\subalign{
                                 1 \le |w| &\le B\\
    |a^2w|&\le B\\
    |c^2w|&\le B\\
    |z^2w| &\le B
             }}} \frac{B}{\max\{|a|,|c|\}} \,\df w\,\df z \\
                               &\le \int_{\substack{\subalign{
                                 1 \le |w| &\le B\\
    |a^2w|&\le B\\
    |c^2w|&\le B
            }}} \frac{B^{3/2}}{\max\{|a|,|c|\}|w|^{1/2}} \,\df w\\
                               &\le \min\rleft\{\frac{B^2}{\max\{|a|,|c|\}^2},
                                 \frac{B^2}{\max\{|a|,|c|\}}
                                 \rright\}\text{,}
  \end{align*}
  which in particular implies the last equality above. We now successively
  replace the sums over $a$ and $c$ by integrals to obtain
  \begin{align*}
    N_{U,H}(B^2) 
    &= \sum_{\ma\ge 1}\frac{\mu(\ma)}{4\zeta(2)\ma}
      \sum_{|c|\ge 1}
      \rleft(\iint_{|a|\ge 1} V_2(\ma a,\ma c,z;B) \, \df a\, \df z + O(\err_3) \rright)+ O(B^2)\\
    &= \sum_{\ma\ge 1}\frac{\mu(\ma)}{4\zeta(2)\ma}
      \rleft( \iiint_{|a|,|c|\ge 1}
      V_2(\ma a,\ma c,z;B) \,\df a\,\df c\,\df z +O(\err_4)\rright) + O(B^2)
    \\
    &= \sum_{\ma\ge 1}\frac{\mu(\ma)}{4\zeta(2)\ma^3}
      \iiint_{|a|,|c|\ge \ma}
      V_2(a,c,z;B) \,\df a\,\df c\,\df z + O(B^2)
  \end{align*}
  since we have
  \begin{align*}
    \err_3 = \max_a \int V_2(\ma a,\ma c,z;B)\,\df z \ll \frac{B^2}{\ma^2|c|^2}
    &&\text{and}&&
                   \sum_{\ma} \frac{1}{\ma}\sum_{c} \frac{B^2}{\ma^2|c|^2} \ll B^2
  \end{align*}
  as well as
  \begin{align*}
    \err_4 = \max_c \iint_{|a|\ge 1} V_2(\ma a, \ma c, z;B)\,\df a\,\df z \ll \frac{B^2}{\ma^2}
    &&\text{and}&&
                   \sum_{\ma}\frac{1}{\ma}\cdot\frac{B^2}{\ma^2} \ll B^2
  \end{align*}
  and moreover the transformations $\ma a \mapsto a$ as well as
  $\ma b \mapsto b$ have been applied for the last equality. Finally, we can
  remove the conditions $|a|,|c| \ge \ma$ from the integral in order to obtain
  $V_3(B)$ since
  \begin{align*}
    &\sum_\ma \frac{1}{\ma^3} 
      \iiint_{|a|,|c| \le \ma} V_2(a,c,z;B) \,\df a\,\df c\,\df z \\
    &\qquad \ll \sum_\ma \frac{1}{\ma^3} 
      \iint_{|a|,|c| \le \ma} \frac{B^2}{\max\{|a|,|c|\}} \,\df a\,\df c\\
    &\qquad \ll \sum_\ma \frac{B^2}{\ma^2} \ll B^2\text{,}
  \end{align*}
  as well as
  \begin{align*}
    &\sum_\ma \frac{1}{\ma^3} 
      \iiint_{|a| \le \ma, |c| \ge \ma} V_2(a,c,z;B) \,\df a\,\df c\,\df z \\
    &\qquad \ll \sum_\ma \frac{1}{\ma^3} 
      \iint_{|a| \le \ma,|c| \ge \ma} \frac{B^2}{|c|^2} \,\df a\,\df c\\
    &\qquad \ll \sum_\ma \frac{1}{\ma^3} 
      \int_{|a| \le \ma} \frac{B^2}{\ma} \,\df a \ll \sum_\ma \frac{B^2}{\ma^3} \ll B^2\text{,}
  \end{align*}
  The case $|a|\ge \ma, |c|\le \ma$ is handled similarly. It follows that we
  obtain
  \begin{align*}
    N_{U,H}(B^2) 
    &=
      \frac{1}{4\zeta(2)} 
      V_3(B)
      +O(B^2)\\
    &=
      \frac{\omega_\infty}{4\zeta(2)\zeta(3)} B^2\log B + O(B^2)\\
    &=
      \frac{\omega_\infty}{8\zeta(2)\zeta(3)} B^2\log(B^2) + O(B^2)
      \text{,}
  \end{align*}
  by Lemma~\ref{lem:volume_real_density}, as predicted in
  Section~\ref{sec:exp-x2}.
\end{proof}

\section{Estimating integral points on the universal torsor of
  $X'_\n$}\label{sec:summations-b}

We are going to prove Theorem~\ref{thm:b_intro} by showing
\begin{align*}
  N_{U',H'}(B^{\frac{1}{\n+1}}) = \cfr B^{\frac{2}{\n+3}} + O_\epsilon(B^{\frac{1}{\n+1}+\epsilon})
\end{align*}
for $\n \ge 2$ and any $\epsilon>0$, where $\cfr$ is as in Section~\ref{sec:exp-b-ge-2}.

As in the preceding section, we repeatedly use \cite[Lemma~3.1]{der09} and
\cite[Lemma~3.6]{df14} to approximate sums by integrals. We define
\begin{equation*}
  \cond(b,d,w,a,c,y,z,t) \coloneqq \max |\Ms_\n'(a,b,c,d,y,z,t,w)|\text{.}
\end{equation*}
Moreover, we define
\begin{align*}
  V_{1,\lambda}(a,c,y,z,t,w; B) \coloneqq
  \begin{dcases}
    \int_{\cond(\lambda b, (\lambda bc+y^\n z^{\n+1}w)/a, ...) \le B} \frac{1}{|a|}\, \df b & \text{for $a\ne 0$,} \\
    \int_{\cond((\lambda ad-y^\n z^{\n+1}w)/c, \lambda d, ...) \le B}
    \frac{1}{|c|}\, \df d & \text{for $c\ne 0$.}
  \end{dcases}
\end{align*}
Note that for $ac\ne 0$ the two cases coincide and that we have
\begin{align*}
  V_{1,1}(a,c,y,z,t,w; B) \ll \frac{B^{1/(\n+3)}}{\max\{|a|,|c|\}^{(\n+5)/(\n+3)}|y|^{2/(\n+3)}}\text{.}
\end{align*}
We also define
\begin{align*}
  V_2(a,c,y,z,t; B) \coloneqq \int
  V_{1,1}(a,c,y,z,t,w; B) \,\df w\text{.}
\end{align*}

\begin{theorem}
  \label{thm:bpc}
  Let $\n \ge 2$. For any $\epsilon > 0$, we have
  \begin{align*}
    N_{U',H'}(B^{\frac{1}{\n+1}}) = \cfr B^{\frac{2}{\n+3}} + O_\epsilon(B^{\frac{1}{\n+1}+\epsilon})\text{.}
  \end{align*}
\end{theorem}
\begin{proof}
  By Corollary~\ref{cor:countingproblem_b}, we have that
  $16\cdot N_{U',H'}(B^{\frac{1}{\n+1}})$ is equal to
  \begin{align*}
    \sum_{\substack{|a|,|c|\ge 0\\
    |y|,|z|,|t|,|w|\ge 1\\
    \gcd(a,c)=\gcd(y,w)=1\\
    \gcd(y,t)=\gcd(z,t)=1
    }}
    \#\rleft\{(b, d) \in \Zd^2 : \pbox{\textwidth}{$\begin{aligned}ad-bc &= y^\n z^{\n+1}w\\\cond(...) &\le B\\\gcd(b,d,z)&=\gcd(b,d,w)=1\end{aligned}$}\rright\}\text{.}
  \end{align*}
  We apply a M\"obius inversion to the conditions
  \begin{align*}
    \gcd(b,d,z) = \gcd(b,d,w) = 1\text{,}
  \end{align*}
  so that, after using the transformation
  $(b, d) \mapsto ([\ma, \mb] b, [\ma, \mb] d)$, we obtain that
  $16 \cdot N_{U',H'}(B^{\frac{1}{\n+1}})$ is equal to
  \begin{align*}
    &\sum_{\substack{|a|,|c| \ge 0\\
    |y|,|z|,|t|,|w|\ge 1\\
    \gcd(a,c)=\gcd(y,w)=1\\
    \gcd(y,t)=\gcd(z,t)=1\\
    }}
    \sum_{\substack{\ma,\mb>0\\\ma | z\\\mb | w}} \mu(\ma)\mu(\mb)
    \#\rleft\{(b, d) \in \Zd^2 : 
    \pbox{\textwidth}{$\begin{aligned}ad-bc = y^\n z^{\n+1}w/[\ma,\mb]\\ \cond([\ma,\mb] b, [\ma,\mb] d, ...) \le B\end{aligned}$}\rright\}\\
    =&
       \sum_{\substack{|a|,|c| \ge 0\\
    |y|,|z|,|t|,|w|\ge 1\\
    \gcd(a,c)=\gcd(y,w)=1\\
    \gcd(y,t)=\gcd(z,t)=1\\
    \mathclap{\cond(0,0,...) \le B}
    }}
    \sum_{\substack{\ma,\mb > 0\\\ma | z\\\mb | w}} \mu(\ma)\mu(\mb)
    \rleft( V_{1,[\ma,\mb]}(a,c,y,z,t,w; B) + O(1)\rright)\\
    =&
       \sum_{\substack{|a|,|c| \ge 0\\
    |y|,|z|,|w|,|t|\ge 1\\
    \gcd(a,c)=\gcd(y,w)=1\\
    \gcd(y,t)=\gcd(z,t)=1\\
    \mathclap{\cond(0,0,...) \le B}
    }}
    \sum_{\substack{\ma,\mb>0\\\ma | z\\ \mb | w}} \mu(\ma)\mu(\mb)
    \rleft( \frac{1}{[\ma,\mb]}V_{1,1}(a,c,y,z,t,w; B) + O(1)\rright)\\
    =&
       \sum_{\substack{|a|,|c| \ge 0\\
    |y|,|z|,|t|,|w|\ge 1\\
    \gcd(a,c)=\gcd(y,w)=1\\
    \gcd(y,t)=\gcd(z,t)=1\\
    }}
    \sum_{\substack{\ma,\mb>0\\\ma | z\\ \mb |w}} \frac{\mu(\ma)\mu(\mb)}{[\ma,\mb]}
    V_{1,1}(a,c,y,z,t,w; B) + O(B^{\frac{1}{\n+1}+\epsilon})
  \end{align*}
  for any $\epsilon > 0$, where we have used the fact that
  $\cond(0,0,\dots) \le B$ implies
  \begin{equation*}
    \max\{|a|,|c|\}^{\n^2+\n}|y|^{\n+1}|zt|^{2\n+2}|w|^{\n+3} \le B
  \end{equation*}
  to obtain the estimate
  \begin{align*}
    \sum_{a,c,y,z,t,w} \sum_{\substack{\ma | z\\\mb |w}} 1 &\ll 
    \sum_{a,c,y,z,t,w} 2^{\omega(z)+\omega(w)} \\
    & \ll
    \sum_{a,c,y,z,t} 2^{\omega(z)} \frac{B^{1/(\n+3)} \log B}{\max\{|a|,|c|\}^{(\n^2+\n)/(\n+3)}|y|^{(\n+1)/(\n+3)}|zt|^{(2\n+2)/(\n+3)}} \\
    & \ll \sum_{a,c} \frac{B^{1/(\n+1)} \log B}{\max\{|a|,|c|\}^{(\n^2+3\n)/(\n+3)}} 
    \ll B^{\frac{1}{\n+1}} (\log B)^2\text{.}
  \end{align*}
  
  We now apply a M\"obius inversion to the condition $\gcd(y,w)=1$, so that,
  after using the transformation $w \mapsto [\mb, \mc]w$, we obtain that
  $16 \cdot N_{U',H'}(B^{\frac{1}{\n+1}})$ is equal to
  \begin{align*}
    \sum_{\substack{|a|,|c| \ge 0\\
    |y|,|z|,|t|,|w|\ge 1\\
    \gcd(a,c)=\gcd(yz,t)=1\\
    }}
    \sum_{\substack{\ma,\mb,\mc>0\\\ma | z \\ \mc | y}} \frac{\mu(\ma)\mu(\mb)\mu(\mc)}{[\ma, \mb]}
    V_{1,1}(a,c,y,z,t,[\mb,\mc] w; B) + O(B^{\frac{1}{\n+1}+\epsilon})\text{.}
  \end{align*}
  Replacing the sum over $w$ by an integral, we obtain
  \begin{align*}
    \sum_{|w| \ge 1} V_{1,1}(a,c,y,z,t,[\mb,\mc] w; B)
    = \int_{|w|\ge 1}V_{1,1}(a,c,y,z,t, [\mb, \mc]w; B) \,\df w + O(R_1)\text{,}
  \end{align*}
  where
  \begin{align*}
    R_1 = \max_w V_{1,1}(a,c,y,z,t, [\mb, \mc]w; B) \ll \frac{B^{1/(\n+3)}}{\max\{|a|,|c|\}^{(\n+5)/(\n+3)}|y|^{2/(\n+3)}}\text{.}
  \end{align*}
  Using the fact that $\cond(0,0,[\mb, \mc]w,\dots) \le B$ and $|w| \ge 1$ imply
  \begin{gather*}
    \max\{|a|,|c|\}^{\n^2+2\n+1}|z|^{\n+1}|\mb|^{\n+3}|t|^{\n^2+3\n+2} \le B\text{,}\\
    |y|^{\n^2+2\n+1}|z|^{\n^2+3\n+2}|\mb|^{\n+3}|t|^{\n+1} \le B\text{,}
  \end{gather*}
  we obtain the estimate
  \begin{align*}
    \sum_{a,c,y,z,t} \sum_{\substack{\ma | z \\ \mb \\ \mc | y}}\frac{R_1}{[\ma, \mb]} &\ll
    \sum_{a,c,y,z,t} \sum_{\mb} \frac{2^{\omega(z)+\omega(y)}B^{1/(\n+3)}}{\mb\max\{|a|,|c|\}^{(\n+5)/(\n+3)}|y|^{2/(\n+3)}}\\
    & \ll \sum_{z,t} \sum_{\mb} \frac{2^{\omega(z)}B^{(1+2/(\n+1))/(\n+3)}\log B}{\mb^{(\n+3)/(\n+1)}|zt|}\\
    & \ll B^{\frac{1+2/(\n+1)}{\n+3}}(\log B)^4
  \end{align*}
  Hence we obtain that $16 \cdot N_{U',H'}(B^{\frac{1}{\n+1}})$ is equal to
  \begin{align*}
    \sum_{\substack{|a|,|c| \ge 0\\
    |y|,|z|,|t| \ge 1\\
    \gcd(a,c)=1\\\gcd(yz,t)=1\\
    }}
    \quad \sum_{\substack{\mathclap{\ma,\mb,\mc>0}\\ \ma | z \\ \mc | y}} \frac{\mu(\ma)\mu(\mb)\mu(\mc)}{[\ma, \mb][\mb, \mc]}
    \int_{|w|\ge [\mb, \mc]}V_{1,1}(a,c,y,z,t, w; B) \,\df w + O(B^{\frac{1}{\n+1}+\epsilon})
  \end{align*}
  
  Removing the condition $|w| \ge [\mb, \mc]$, we obtain
  \begin{align*}
    \int_{|w|\ge [\mb, \mc]}V_{1,1}(a,c,y,z,t, w; B) \,\df w = V_2(a,c,y,z, t; B) + O(R_2)\text{,}
  \end{align*}
  where, using the geometric mean of the conditions 
  \begin{equation*}
    \max\{|a|,|c|\}^{\frac{\n^2+2\n+1}{2}}|y|^{\frac{\n^2+2\n+1}{2}}|zt|^{\frac{\n^2+4\n+3}{2}}|w|^{\n+3} \le B\text{,}
  \end{equation*}
  (implied by $\cond(0,0,\dots) \le B$) with weight
  $\delta\coloneqq\frac{2}{\n+1}+\epsilon(\n+3)$ and of $|w| \le [\beta,\gamma]$
  with weight $1-\delta$,
  \begin{align*}
    R_2 &= \int_{|w|\le [\mb, \mc]}V_{1,1}(a,c,y,z,t, w; B) \,\df w \\
        &\ll \frac{[\mb,\mc]^{1-\delta}B^{1/(\n+1)+\epsilon}}
          {\max\{|a|,|c|\}\max\{|ay|,|cy|\}^{1+\epsilon(\n^2+2\n+1)/2}|zt|^{1+\epsilon(\n^2+4\n+3)/2}}
  \end{align*}
  for every sufficiently small $\epsilon > 0$. Summing $R_2$ over the remaining
  variables gives the error term
  \begin{equation*}
    \sum_{a,c,y,z,t}\sum_{\substack{\ma|z\\\mb\\\mc|y}} \frac{|\mu(\alpha)\mu(\beta)\mu(\gamma)|R_2}{[\ma, \mb][\mb, \mc]}
    \ll_\epsilon \sum_{\mb}\frac{B^{1/(\n+1)+\epsilon}}{\mb^{1+\delta}}\\
    \ll_\epsilon B^{\frac{1}{\n+1}+\epsilon} \text{.}
  \end{equation*}
  Hence $16 \cdot N_{U',H'}(B^{\frac{1}{\n+1}})$ is equal to
  \begin{align*}
    &
      \sum_{\substack{|a|,|c|\ge0\\|y|,|z|,|t| \ge 1\\
    \gcd(a,c)=\gcd(yz,t)=1\\
    }}
    \sum_{\substack{\ma,\mb,\mc>0\\ \ma | z \\ \mc | y}} \frac{\mu(\ma)\mu(\mb)\mu(\mc)}{[\ma, \mb][\mb, \mc]}
    V_{2}(a,c,y,z,t; B) + O_\epsilon(B^{\frac{1}{\n+1}+\epsilon})\\
    =&
       \sum_{\substack{|a|,|c|\ge0\\|y|,|z|,|t| \ge 1\\
    \gcd(a,c)=\gcd(yz,t)=1\\
    }}
    \sum_{\substack{\ma,\mb,\mc>0\\ \ma | z \\ \mc | y}} \frac{\mu(\ma)\mu(\mb)\mu(\mc)}{[\ma, \mb][\mb, \mc]}
    V_{2}(a,c,y,z,t; 1) B^{\frac{2}{\n+3}} + O_\epsilon(B^{\frac{1}{\n+1}+\epsilon})
    \text{,}
  \end{align*}
  where we have applied the transformation
  \begin{equation*}
    (b, w) \mapsto B^{\frac{1}{\n+3}}(b, w) \text{ or }
    (d, w) \mapsto B^{\frac{1}{\n+3}}(d, w)
  \end{equation*}
  inside the integral.
  
  Next, we apply the transformation,
  \begin{equation*}
    (at, ct, yz, z, t) \mapsto (\ca,\cc,\cy, z, t)
  \end{equation*}
  and then
  \begin{equation*}
    (b, w) \mapsto ((zt)^{\frac{2}{\n+3}}b, (zt)^{\frac{-\n-1}{\n+3}}w) \text{ or }
    (d, w) \mapsto ((zt)^{\frac{2}{\n+3}}d, (zt)^{\frac{-\n-1}{\n+3}}w)
  \end{equation*}
  inside the integral to obtain that $16 \cdot N_{U',H'}(B^{\frac{1}{\n+1}})$ is equal to 
  \begin{align*}
    &
      \sum_{\substack{|\ca|,|\cc| \ge 0\\|\cy|,|z|,|t| \ge 1\\
    t | \gcd(\ca, \cc)\\
    z | \cy\\
    \gcd(\ca/t,\cc/t)=1\\\gcd(\cy,t)=1\\
    }}
    \sum_{\substack{\ma,\mb,\mc>0\\ \ma | z \\ \mc | \cy/z}} \frac{\mu(\ma)\mu(\mb)\mu(\mc)}{[\ma, \mb][\mb, \mc]}
    V_{2}(\ca/t,\cc/t,\cy/z, z, t; 1) B^{\frac{2}{\n+3}} + O_\epsilon(B^{\frac{1}{\n+1}+\epsilon})\\
    =&
       \sum_{\substack{|\ca|,|\cc| \ge 0 \\ |\cy| \ge 1 \\ \gcd(\ca,\cc,\cy) = 1 \\ (\ca, \cc) \ne (0,0)
    }} \vartheta(\ca,\cc,\cy)
    B^{\frac{2}{\n+3}} + O_\epsilon(B^{\frac{1}{\n+1}+\epsilon})\text{,}
  \end{align*}
  where 
  \begin{align*}
    \vartheta(\ca,\cc,\cy)  &\coloneqq
    \sum_{\substack{|z|,|t| \ge 1\\\mathclap{|t| = \gcd(\ca,\cc)}\\z | \cy\\ \ma,\mb,\mc>0\\ \ma | z \\ \mc | \cy/z}} \frac{\mu(\ma)\mu(\mb)\mu(\mc)}
    {[\ma,\mb][\mb,\mc]}
    |z|^{\frac{-\n+1}{\n+3}}|t|^{\frac{4}{\n+3}} V_{2}(\ca,\cc,\cy,1,1; 1)\\
    &= 2|\gcd(\ca,\cc)|^{\frac{4}{\n+3}}\sum_{\substack{|z|\ge 1\\z | \cy\\ \ma,\mb,\mc>0\\ \ma | z \\ \mc | \cy/z}} \frac{\mu(\ma)\mu(\mb)\mu(\mc)}
    {[\ma,\mb][\mb,\mc]}
    |z|^{\frac{-\n+1}{\n+3}} V_{2}(\ca,\cc,\cy,1,1; 1)\\
    &= 4\omega_{\infty,(\ca:\cc:\cy)}\prod_{p\text{ prime}}\lambda_p\omega_{p,(\ca:\cc:\cy)}\text{.}
  \end{align*}
  In total,
  \begin{align*}
    N_{U',H'}(B^{\frac{1}{\n+1}}) &= \sum_{\substack{|\ca|,|\cc| \ge 0 \\ |\cy| \ge 1 \\ \gcd(\ca,\cc,\cy) = 1 \\ (\ca, \cc) \ne (0,0)}}
    \rleft(\frac{1}{4}\omega_{\infty,(\ca:\cc:\cy)}\prod_{p\text{ prime}}\lambda_p\omega_{p,(\ca:\cc:\cy)}\rright) B^{\frac{2}{\n+3}} + O_\epsilon(B^{\frac{1}{\n+1}+\epsilon})\\
    &=
    \sum_{\substack{x \in \Pd^2(\Qd)\setminus(\Vd(\ca,\cc)\cup\Vd(\cy))}}
    \rleft(\frac{1}{2}\omega_{\infty,x}\prod_{p\text{ prime}}\lambda_p\omega_{p,x}\rright) B^{\frac{2}{\n+3}} + O_\epsilon(B^{\frac{1}{\n+1}+\epsilon})\text{,}
  \end{align*}
  as predicted in Section~\ref{sec:exp-b-ge-2}.
\end{proof}

\begin{remark}
  We have omitted the details of the calculation of $\vartheta(\ca,\cc,\cy)$
  since, according to \cite[Corollaire~6.2.18]{MR1340296}, Manin's conjecture is
  true with Peyre's constant for all heights on $\Pd^1_\Qd \cong \fibphibx$ and
  hence it follows that $\vartheta(\ca,\cc,\cy)$ is equal to
  $2\cfr_{(\ca:\cc:\cy)}$.
\end{remark}

\begin{prop}
  \label{prop:cxb-finite}
  We have 
  \begin{align*}
    \sum_{\substack{x \in \Pd^2(\Qd)\setminus(\Vd(\ca,\cc)\cup\Vd(\cy))}} \cfr_x < \infty\text{.}
  \end{align*}
\end{prop}
\begin{proof}
  In the case $\ca\ne 0$, the condition $\max|\Ms_\n'(\ca,b,\cc,(b\cc+\cy^\n w)/\ca,\cy,1,1,w)| \le 1$ implies
  \begin{equation*}\label{eqn:c}
    |b|^{\n+1}|\ca|^{2\n+2}|w|^2 \le 1 \text{ and } 
    |\ca|^{\n^2+2\n+1}|w|^{\n+3} \le 1\text{,}\tag{$*$}
  \end{equation*}
  hence we obtain
  \begin{align*}
    V_2(\ca,\cc,\cy,1,1; 1) \le \iint_{\eqref{eqn:c}}\frac{1}{|\ca|}\,\df b\, \df w
    \ll \int_{\eqref{eqn:c}} \frac{1}{|\ca|^3|w|^{2/(\n+1)}} \,\df w
    \ll \frac{1}{|\ca|^{3+(\n^2-1)/(\n+3)}}\text{.}
  \end{align*}
  Similarly, we obtain
  \begin{align*}
    V_2(\ca,\cc,\cy,1,1; 1) \le \frac{1}{|\cc|^{3+(\n^2-1)/(\n+3)}}\text{.}
  \end{align*}
  The condition $\max|\Ms_\n'(\ca,b,\cc,(b\cc+\cy^\n w)/\ca,\cy,1,1,w)| \le 1$
  also implies
  \begin{equation*}\label{eqn:cc}
    |b|^{\n+3}|\ca|^2|\cy|^2 \le 1 \text{ and } 
    |(b\cc+\cy^\n w)/\ca|^{\n+3}|\ca|^2|\cy|^2 \le 1\text{,}\tag{$**$}
  \end{equation*}
  hence we obtain
  \begin{align*}
    V_2(\ca,\cc,\cy,1,1; 1) &\le \iint_{\eqref{eqn:cc}}\frac{1}{|\ca|}\,\df b\, \df w
                              = \int_{\max\{|b|,|w|\}^{\n+3}|\ca\cy|^2 \le 1} \frac{1}{|\cy|^\n}\, \df b\, \df w\\
                            &\ll \frac{1}{|\ca|^{4/(\n+3)}|\cy|^{\n+4/(\n+3)}}\text{.}
  \end{align*}
  Similarly, we obtain
  \begin{align*}
    V_2(\ca,\cc,\cy,1,1; 1) \ll \frac{1}{|\cc|^{4/(\n+3)}|\cy|^{\n+4/(\n+3)}}\text{.}
  \end{align*}
  Together, we obtain
  \begin{align*}
    V_2(\ca,\cc,\cy,1,1; 1) \ll \frac{1}{\max\{|\ca|,|\cc|\}^{4/(\n+3)}\max\{|\ca|,|\cc|,|\cy|\}^{(\n^2+3\n+4)/(\n+3)}}\text{.}
  \end{align*}
  There exist $\lambda_1, \lambda_2 > 0$ with
  \begin{align*}
    \lambda_1 > \frac{2\n+2}{\n^2+3\n+4}\text{,}&& \lambda_2 > \frac{\n+3}{\n^2+3\n+4}\text{,}&& \lambda_1+\lambda_2=1\text{,}
  \end{align*}
  so that using $\max\{|\ca|,|\cc|,|\cy|\} \ge \max\{|\ca|,|\cc|\}^{\lambda_1}|\cy|^{\lambda_2}$, we obtain that there exists $\mu > 0$
  such that
  \begin{align*}
    V_2(\ca,\cc,\cy,1,1; 1) \ll \frac{1}{\max\{|\ca|,|\cc|\}^{2+\mu}|\cy|^{1+\mu}}\text{.}
  \end{align*}
  With $\vartheta(\ca,\cc,\cy)$ from the proof of Theorem~\ref{thm:bpc}, we have
  \begin{align*}
    \sum_{\substack{\ca,\cc\ge 0\\ \cy \ge 1\\ (\ca,\cc)\ne(0,0)}} \vartheta(\ca,\cc,\cy) &\ll \sum_{\ca,\cc} \frac{\gcd(\ca,\cc)}{\max\{\ca,\cc\}^{2+\mu}}
    \sum_{\cy} \frac{4^{\omega(\cy)}d(\cy)}{\cy^{1+\mu}}\text{.}
  \end{align*}
  Our aim is to show that this sum converges. We have
  \begin{align*}
    \sum_{\ca \le \cc \le M} \frac{\gcd(\ca,\cc)}{\max\{\ca,\cc\}^{2+\mu}}
    &= \sum_{\cc \le M} \frac{1}{\cc^{2+\mu}} \sum_{\ca\le\cc} \gcd(\ca, \cc)
      =  \sum_{\cc \le M} \frac{1}{\cc^{1+\mu}} \sum_{d|\cc} \frac{\phi(d)}{d}\\
    &\ll \sum_{\cc \le M} \frac{1}{\cc^{1+\mu}} \sum_{d|\cc} 1
      = \sum_{\cc \le M} \frac{d(\cc)}{\cc^{1+\mu}}\\ 
    &\ll \sum_{\cc\le M}\frac{d(\cc)}{M^{1+\mu}}\ + \int_1^M \sum_{\cc \le \lambda}\frac{d(\cc)}{\lambda^{2+\mu}}\,\df \lambda\\
    &\ll \frac{\log M}{M^{\mu}} + \int_1^M \frac{\log \lambda}{\lambda^{1+\mu}} \,\df \lambda
      \ll 1\text{.}
  \end{align*}
  Note that we have
  \begin{align*}
    \sum_{\cy \le M}4^{\omega(\cy)}d(\cy) = \sum_{\cy \le M}\sum_{z|\cy}4^{\omega(\cy)} &= \sum_{\substack{y,z\ge 1\\yz \le M}} 4^{\omega(yz)} \ll \sum_{y \le M} 4^{\omega(y)} 
    \sum_{z \le M/y} 4^{\omega(z)} \\&\ll \sum_{y \le M} \frac{4^{\omega(y)}M(\log M)^3}{y} \ll M (\log M)^7\text{.} 
  \end{align*}
  It follows that we have
  \begin{align*}
    \sum_{\cy\le M} \frac{4^{\omega(\cy)}d(\cy)}{\cy^{1+\mu}}
    &\ll \sum_{\cy\le M} \frac{4^{\omega(\cy)}d(\cy)}{M^{1+\mu}} + \int_1^M \sum_{\cy\le \lambda}\frac{4^{\omega(\cy)}d(\cy)}{\lambda^{2+\mu}} \,\df \lambda\\
    &\ll \frac{(\log M)^7}{M^\mu} + \int_1^M \frac{(\log M)^7}{M^{1+\mu}} \, \df \lambda \ll 1\text{.}\qedhere
  \end{align*}
\end{proof}

\bibliographystyle{amsalpha}
\bibliography{manin_spherical}

\end{document}